\documentclass[a4paper]{article}
\usepackage{geometry}
\usepackage{amssymb}
\usepackage{latexsym}
\usepackage{amsmath}
\usepackage{amsfonts}
\usepackage{indentfirst}
\usepackage{graphicx}
 \usepackage[colorlinks=true]{hyperref}
\usepackage{mathrsfs} 
\usepackage{chngcntr}
\usepackage{color}

 \usepackage{mathrsfs}

\newtheorem{Thm}{Theorem}[section]
\newtheorem{Cor}[Thm]{Corollary}
\newtheorem{Lem}[Thm]{Lemma}
\newtheorem{Pro}[Thm]{Proposition}

\newtheorem{Rek}[Thm]{Remark}
\newtheorem{Def}[Thm]{Definition}

\numberwithin{equation}{section}

\newenvironment{proof}{\medskip\par\noindent{\bf Proof\/}:\quad}{\qquad
\raisebox{-0.5mm}{\rule{1.5mm}{1mm}}\vspace{6pt}}

\begin{document}
\title{Nonrelativistic limit of nonlinear Dirac ground states on noncompact quantum graphs}

\author{Guangze Gu, Michael Ruzhansky, Guoyan Wei, Zhipeng Yang\thanks{Corresponding author:yangzhipeng326@163.com.}}

\date{}  

\AtEndDocument{%
	\par
	\bigskip
	\bigskip
	
	\noindent
	\textbf{Guangze Gu:}\\[0.2em]
	\textsc{Department of Mathematics, Yunnan Normal University, Kunming, China}\\
	\textsc{Yunnan Key Laboratory of Modern Analytical Mathematics and Applications, Kunming, China}\\
	[0.3em]
	\textit{E-mail address}: \texttt{guangzegu@163.com}\\[1.5em]
	\noindent
	\textbf{Michael Ruzhansky:}\\[0.2em]
	\textsc{Department of Mathematics: Analysis, Logic and Discrete Mathematics, Ghent University, Belgium}\\
	\textsc{School of Mathematical Sciences, Queen Mary University of London, United Kingdom}\\
	[0.3em]
	\textit{E-mail address}: \texttt{Michael.Ruzhansky@UGent.be}\\[1.5em]
	\noindent
	\textbf{Guoyan Wei:}\\[0.2em]
	\textsc{Department of Mathematics, Yunnan Normal University, Kunming, China}\\[0.3em]
	\textit{E-mail address}: \texttt{guoyanwei1005@163.com}\\[1.5em]
	\noindent
	\textbf{Zhipeng Yang:}\\[0.2em]
	\textsc{Department of Mathematics, Yunnan Normal University, Kunming, China}\\
	\textsc{Department of Mathematics: Analysis, Logic and Discrete Mathematics, Ghent University, Belgium}\\[0.3em]
	\textit{E-mail address}: \texttt{yangzhipeng326@163.com}%
}

\maketitle
\begin{abstract}
We study the nonrelativistic limit of nonlinear Dirac ground states on finite noncompact metric graphs with Kirchhoff-type conditions. For a nondecreasing nonlinearity of arbitrary finite superquadratic growth, a sufficiently long pendant path guarantees ground states whose nonrelativistic limits are Schr\"odinger ground states. Every choice of these states is relatively compact in $H^1$ and satisfies a uniform exponential decay estimate. On general graphs, we prove a compactness alternative: loss of compactness requires at least the action of a full-line ground state. Translated line profiles show that the loss constants are sharp. The proof combines an $H^1$ estimate in terms of the action with convergence of Dirac fiber maxima to Schr\"odinger ray maxima. For power nonlinearities, we give an integral criterion in terms of the dimensionless pendant length and obtain the same conclusions for arbitrary edge lengths on four graph topologies.
\end{abstract}


\section{Introduction and main results}\label{Sec1}
We study the existence of bound states and the nonrelativistic limit of the nonlinear Dirac equation (NLDE) on a noncompact metric quantum graph $\mathcal{G}$:
\begin{equation}\label{eq1.1}
-i c \frac{d}{d x} \sigma_1 \psi+m c^2 \sigma_3 \psi-\omega \psi=g(|\psi|) \psi, \quad \text { in } \mathcal{G},
\end{equation}
where $\psi: \mathcal{G} \rightarrow \mathbb{C}^2$ is a spinor field, and $g:[0,\infty)\rightarrow\mathbb R$ is a nonlinearity satisfying the assumptions below. Here, $c>0$ denotes the speed of light, $m>0$ is the mass of the particle, and $\omega\in\mathbb R$ is a frequency parameter. In the variational existence theory we shall work in the spectral gap $\omega\in(-mc^2,mc^2)$.
\par
We are particularly interested in the regime where the characteristic velocities of the system are much smaller than the speed of light $c$, so that relativistic effects can be neglected.
The matrices $\sigma_1$ and $\sigma_3$ are the standard Pauli-Dirac matrices, defined by
$$
\sigma_1=\left(\begin{array}{ll}
	0 & 1 \\
	1 & 0
\end{array}\right), \quad \sigma_3=\left(\begin{array}{cc}
	1 & 0 \\
	0 & -1
\end{array}\right).
$$
\par
Our main result concerns graphs containing a long pendant path, on which we construct action ground states and study their nonrelativistic limit.
More precisely, we fix $m>0$ and let both $c, \omega \rightarrow \infty$ with the frequency shift specified below. Along a subsequence, the ground states $\psi_n=(u_n,v_n)^T$ constructed below satisfy $\psi_n\to(u,0)^T$, where $u\ne0$ solves the nonlinear Schr\"odinger equation:
\begin{equation}\label{eq1.2}
-\Delta u-\nu u=2 m g(|u|) u
\end{equation}
Here $u:\mathcal G\to\mathbb C$ and $\nu<0$ is a constant.
We also prove bounds and exponential decay estimates that are uniform in $c$.

In recent decades, substantial attention has been devoted to the study of the nonlinear Dirac equation and its connection to systems of coupled nonlinear Schr\"odinger equations (NLSE). The NLDE arises as a relativistic model describing self-interacting spinor fields, with applications ranging from quantum field theory to condensed matter physics (see \cite{MR187642,MR1897689}).
A standard higher-dimensional form is
\begin{equation}\label{eq1.3}
	i \hbar \partial_t \hat{\psi}=-i c \hbar \sum_{k=1}^3 \alpha_k \partial_k \hat{\psi}+m c^2 \beta \hat{\psi}-G_{\hat{\psi}}(x, \hat{\psi}),
\end{equation}
where $\alpha_k$ and $\beta$ are the Dirac matrices and $\hat{\psi}$ is a four-component spinor, $c$ denotes the speed of light, $m>0$ represents the mass of the generic particle of the system, and $\hbar$ is Planck's constant. Assuming that $G\left(x, e^{i \theta} \hat{\psi}\right)=G(x, \hat{\psi})$ for all $\theta \in[0,2 \pi]$, a standing wave solution of \eqref{eq1.3} is a solution of the form $\hat{\psi}(t, x)=e^{-\frac{i \omega t}{\hbar}} \psi(x)$. The one-dimensional two-component reduction, with $\hbar=1$, $\alpha_1=\sigma_1$, $\beta=\sigma_3$ and radial nonlinearity $G_{\psi}=g(|\psi|)\psi$, leads to \eqref{eq1.1}.
Esteban and Séré \cite{MR1344729} developed a variational approach to stationary nonlinear Dirac equations. Subsequent work established existence and decay estimates under various assumptions.
We refer the reader to
\cite{MR2232435,MR4182313,MR4414162,MR4799303,MR2389415,MR2434900,MR3023429,MR4523529,MR4290374,MR2468430} and their references.
\par
Nonrelativistic limits have been studied for several nonlinear Dirac systems.
For example, Esteban and Séré \cite{MR1869528} described a limiting process that shows how solutions of Dirac-Fock equations converge toward the corresponding solutions of Hartree-Fock equations when the speed of light tends to infinity.
Recently, Ding et al. \cite{MR4728777,MR4280523,MR4610809} studied the nonrelativistic limit and some properties of solutions for the stationary Dirac equation with both local and nonlocal nonlinearities.
In particular, Borrelli, Carlone and Tentarelli \cite{MR3934110} proved existence, multiplicity and a nonrelativistic limit for nonlinearities localized on the compact core. Their later work \cite{MR4200758} treats the global power nonlinearity and constructs standing waves on star graphs near the upper spectral threshold. 

For a broader account of standing waves on quantum graphs, including existence and stability, see \cite{MR4438617}.
For scalar power nonlinearities, minimization on the Nehari manifold depends on the graph topology. Condition~(H) means that at each point there are two injective curves starting there, parametrized by arc length on $[0,\infty)$, whose images have at most countably many common points. The Nehari infimum then equals the line level and is attained only on the line and the exceptional towers of bubbles \cite[Definition~2.6 and Theorem~2.7]{DeCoster2023}. Minimization among nonzero solutions is a different problem: star graphs can have least action solutions even when the Nehari infimum is not attained \cite[Section~4.3]{DeCoster2023}. Graphs without terminal vertices may also have Nehari ground states; the tadpole is an example \cite[Theorem~4.8]{DeCoster2026}. A sufficiently long pendant gives a scalar strict comparison with the line \cite[Theorem~4.12]{DeCoster2026}.

For the nonlinearity, set
\[
G(s):=\int_0^s g(t)t\,dt,
\qquad
\hat G(s):=\frac12 g(s)s^2-G(s),
\qquad s\ge0.
\]
We use the following conditions:
\begin{itemize}
	\item [$(g_1)$] $g \in C^1([0, \infty))$;
	\item [$(g_2)$] $g(s) \rightarrow 0$ as $s \rightarrow 0$;
	\item [$(g_3)$] There exist $p \in(2, \infty)$ and $C_1>0$ such that
	\[
	g(s) \leq C_1\left(1+s^{p-2}\right), \qquad \forall s\ge0;
	\]
	\item [$(g_4)$] There exists $\theta>2$ such that
	\[
	0<\theta \cdot G(s) \leq g(s) s^2, \qquad \forall s>0;
	\]
	\item [$(g_5)$] $\hat G(s)>0$ for every $s>0$, and there exist $c_1>0$ and $R>0$ such that
	\[
	\hat G(s)\ge c_1 s^p, \qquad \forall s\ge R;
	\]
	\item [$(g_6)$] $g$ is nondecreasing on $[0,\infty)$.
\end{itemize}
Throughout we assume $(g_1)$--$(g_5)$; condition $(g_6)$ is imposed where stated.
Let $Y_c=\operatorname{dom}(|\mathcal D_c|^{1/2})$ be the form domain, with positive and negative spectral subspaces $Y_c^\pm$, and let $\Phi_{c,\omega}$ be the action in \eqref{eq2.11}. For $e\in Y_c^+\setminus\{0\}$ put
\[
\Gamma_{c,\omega}(e)=\sup_{t\ge0,\,\eta\in Y_c^-}\Phi_{c,\omega}(te+\eta).
\]
We consider sequences of parameters satisfying
\begin{equation}\label{eq1.4}
c_n\to+\infty,\qquad \omega_n\to+\infty,\qquad c_n>0,
\end{equation}
\begin{equation}\label{eq1.5}
0<\omega_n<mc_n^2,
\end{equation}
\begin{equation}\label{eq1.6}
\omega_n-mc_n^2\to\frac{\nu}{2m},\qquad m>0,\quad \nu<0.
\end{equation}
Write $\Phi_n=\Phi_{c_n,\omega_n}$. Let $\mathcal H$ be the scalar vertex-continuous $H^1$ space and set
\[
J(w)=\frac12\int_{\mathcal G}(|w'|^2-\nu|w|^2)\,dx-2m\int_{\mathcal G}G(|w|)\,dx,
\qquad s_{\mathcal G}=\inf_{w\in\mathcal H\setminus\{0\}}\max_{t\ge0}J(tw).
\]
For fixed $m,c,\omega$, put
\[
d_c=\inf_{e\in Y_c^+\setminus\{0\}}\Gamma_{c,\omega}(e).
\]
An action ground state is a nonzero solution whose action equals $d_c$ for the Dirac equation, or $s_{\mathcal G}$ for the Schr\"odinger equation. Under $(g_6)$, the scalar level is the infimum on the Nehari manifold, as shown in Lemma~\ref{Lem5.1}. Following \cite[Definitions~1.1--1.2]{DeCoster2023}, we call a nonzero solution with least action among all nonzero solutions a least action solution. The ground states constructed below are also least action solutions; see Lemmas~\ref{Lem3.12} and \ref{Lem5.1}.

A pendant path is a terminal metric interval $[0,L]$ whose interior has no branching and which meets the rest of the graph only at $L$. The attachment point may lie in the interior of an edge. Purely transmissive vertices of degree two may be suppressed when identifying this interval. Our main geometric assumption is the presence of a sufficiently long pendant path.

\begin{Thm}\label{Thm1.1}
Assume $(g_1)$--$(g_6)$ with $2<p<\infty$, and fix $m>0$ and $\nu<0$. There is $L_*=L_*(m,\nu,g)>0$ with the following property. Let $\mathcal G$ be a fixed finite connected noncompact metric graph with nonempty compact core $\mathcal K$ and a pendant path of length at least $L_*$. For every sequence satisfying \eqref{eq1.4}--\eqref{eq1.6}, there is $n_0$ such that, for all $n\ge n_0$, the Dirac equation \eqref{eq1.1} has an action ground state of level
\[
d_n=\inf_{e\in Y_{c_n}^+\setminus\{0\}}\Gamma_{c_n,\omega_n}(e),
\qquad d_n\longrightarrow\frac{s_{\mathcal G}}{2m}.
\]
For every choice of these ground states $\psi_n=(u_n,v_n)^T$, $n\ge n_0$, the following conclusions hold:
\begin{itemize}
\item[(i)] $\sup_{n\ge n_0}\|\psi_n\|_{H^1}<\infty$, $\sup_{n\ge n_0}\|\psi_n\|_{L^\infty}<\infty$, and $\|v_n\|_{H^1}=O(c_n^{-1})$;
\item[(ii)] every subsequence has a further subsequence for which $u_n\to u$ strongly in $H^1(\mathcal G)$, where $u\ne0$ is an action ground state of \eqref{eq1.2} with Kirchhoff conditions and $J(u)=s_{\mathcal G}$;
\item[(iii)] there are $C,\kappa>0$, independent of $n$, such that
\[
|\psi_n(x)|\le C e^{-\kappa d(x,\mathcal K)}
\quad(n\ge n_0,\ x\in\mathcal G).
\]
\end{itemize}
The length threshold $L_*$ is independent of the remaining geometry of the graph.
\end{Thm}
The theorem constructs the states and derives their compactness from the pendant geometry. Its proof compares the Dirac fiber maxima with the Schr\"odinger ray maxima. The additional hypothesis $(g_6)$ is used to identify these variational levels with the least action among solutions; the fiber comparison itself holds under $(g_1)$--$(g_5)$.

\begin{Cor}\label{Cor1.2}
Let $g(s)=s^{p-2}$ with $p\ge3$, and let $\mathcal G$ be a fixed finite connected noncompact graph with nonempty compact core. There is a dimensionless constant $\ell_p>0$, defined by the integrals in Proposition~\ref{Pro5.8}, such that the conclusions of Theorem~\ref{Thm1.1} hold whenever $\mathcal G$ has a pendant path of length $L$ satisfying
\begin{equation}\label{eq1.7}
L\sqrt{-\nu}>\ell_p.
\end{equation}
The constant $\ell_p$ is independent of $m$. In particular, every graph in this class with a terminal vertex and a pendant path of length $L>0$ is covered for
$-\nu>(\ell_p/L)^2$.
\end{Cor}
The restriction $p\ge3$ ensures $g\in C^1([0,\infty))$; the scalar length calculation in Proposition~\ref{Pro5.8} holds for every $p>2$.

\begin{Cor}\label{Cor1.3}
Let $g(s)=s^{p-2}$ with $p\ge3$. Let $\mathcal G$ be a line with one pendant, a signpost, a tadpole or a $3$-fork, as in \cite[Figure~6]{DeCoster2026}, with arbitrary positive lengths on its bounded edges. For every $m>0$, $\nu<0$ and parameter sequence \eqref{eq1.4}--\eqref{eq1.6}, action ground states exist for all sufficiently large $n$, and the level convergence and conclusions~(i)--(iii) of Theorem~\ref{Thm1.1} hold.
\end{Cor}
Here a signpost consists of a line joined by a bounded edge to a loop; a tadpole consists of a loop and a half-line meeting at one vertex; and a $3$-fork consists of one half-line and three bounded terminal edges meeting at one vertex. Purely transmissive subdivisions are allowed.

We next describe the possible loss of compactness on general graphs. The line equation has, up to phase and translation, a unique nonzero bound state $U_{c,\omega}$. The scalar equation $-q''-\nu q=2mg(q)q$ has a unique positive even $H^1(\mathbb R)$ solution $q_0$. These facts are proved in Propositions~\ref{Pro2.9} and \ref{Pro2.10}. Define
\[
d_c^\infty(\omega)=\int_{\mathbb R}\hat G(|U_{c,\omega}|)\,dx,
\qquad e_{\mathbb R}(m,\nu)=\int_{\mathbb R}\hat G(q_0)\,dx.
\]
The scalar action satisfies $J_{\mathbb R}(q_0)=2m e_{\mathbb R}(m,\nu)$; an integral formula for $e_{\mathbb R}$ is given in \eqref{eq2.26}.

For fixed $c$ and $\omega$, a sequence $\{\psi_j\}\subset Y_c$ is called a Cerami sequence at a finite level $d\in\mathbb R$, or a $(C)_d$ sequence, if
\[
\Phi_{c,\omega}(\psi_j)\to d,
\qquad (1+\|\psi_j\|_{Y_c})\|\Phi_{c,\omega}'(\psi_j)\|_{Y_c^*}\to0.
\]

\begin{Thm}\label{Thm1.4}
Let $\mathcal G$ be a finite connected noncompact metric graph with nonempty compact core $\mathcal K$, and assume $(g_1)$--$(g_5)$ with $2<p<\infty$.
\begin{itemize}
\item[(a)] Fix $m,c>0$ and $|\omega|<mc^2$. Let $\{\psi_j\}\subset Y_c$ be a Cerami sequence for $\Phi_{c,\omega}$ at a level $d>0$. After extraction, $\psi_j\rightharpoonup\psi$ in $Y_c$, where $\Phi_{c,\omega}'(\psi)=0$ and $0\le\Phi_{c,\omega}(\psi)\le d$. Either $\psi_j\to\psi$ strongly in $Y_c$, with $\psi\ne0$ and $\Phi_{c,\omega}(\psi)=d$, or a further subsequence satisfies
\[
d\ge\Phi_{c,\omega}(\psi)+d_c^\infty(\omega).
\]
In the latter case a nonzero full-line Dirac profile is obtained by translation along a half-line.
\item[(b)] Fix $m>0$ and $\nu<0$. Let $\psi_n=(u_n,v_n)^T$ be nonzero bound states under \eqref{eq1.4}--\eqref{eq1.6}, with $\sup_n\Phi_n(\psi_n)<\infty$. Then
\[
\sup_n\|\psi_n\|_{H^1}<\infty,\qquad \|v_n\|_{H^1}=O(c_n^{-1}).
\]
From every subsequence one can extract a further subsequence, denoted by $n_k$, such that
\[
\Phi_{n_k}(\psi_{n_k})\to\ell,\qquad
u_{n_k}\rightharpoonup u\text{ in }\mathcal H,
\]
where $u$ is a possibly zero Kirchhoff solution of \eqref{eq1.2}. Either
\[
\psi_{n_k}\to(u,0)^T\text{ strongly in }H^1,\quad u\ne0,
\qquad \ell=\int_{\mathcal G}\hat G(|u|)\,dx,
\]
and this extracted subsequence satisfies
$|\psi_{n_k}(x)|\le C e^{-\kappa d(x,\mathcal K)}$ with $C,\kappa>0$ independent of $k$; or, after a further extraction, a nonzero full-line Schr\"odinger profile escapes along a half-line and
\[
\ell\ge\int_{\mathcal G}\hat G(|u|)\,dx+e_{\mathbb R}(m,\nu).
\]
\item[(c)] Both strong convergence and escape occur within this graph class. Equality occurs in the two escape inequalities, so their loss constants cannot be increased. Compactness at the full-line level fails in general.
\end{itemize}
\end{Thm}

For graphs with nonempty compact core, the graph class used here is precisely the class called generalized star graphs in \cite[Definition~2.1]{Yang2026}. A graph obtained from the line by inserting vertices of degree two belongs to this class. Under $(g_6)$, its Dirac ground level equals $d_c^\infty(\omega)$ \cite[Theorem~1.2]{Yang2026}. The translated profiles in Proposition~\ref{Pro4.5} realize the escape in Theorem~\ref{Thm1.4}. Long pendant paths give a sufficient geometric condition for compactness of the ground states constructed in Theorem~\ref{Thm1.1}. Conversely, Proposition~\ref{Pro5.4} shows that compactness of a nonzero solution family with $\limsup_n\Phi_n(\psi_n)\le s_{\mathcal G}/(2m)$ forces attainment of $s_{\mathcal G}$. For powers on nonexceptional graphs satisfying~(H), every such family escapes; after extraction, one line profile carries all its action.

Section~\ref{Sec2} introduces the graph spaces, the Dirac operator and the line profiles. Section~\ref{Sec3} proves linking, compactness below the Dirac line level, and attainment of the ground-state level. Section~\ref{Sec4} proves the nonrelativistic compactness alternative, uniform decay and sharpness. Section~\ref{Sec5} compares the variational levels and derives the ground-state limits from geometric conditions. 

\par
\textbf{Notation.}~We use the following notation.
\begin{itemize}
	\item[$\bullet$] For $1<q<\infty$, $q^{\prime}=q/(q-1)$ denotes the conjugate exponent; for $q=1$, we set $q^{\prime}=\infty$.
	\item[$\bullet$] We denote the usual norm on $L^t(\mathcal G)$ by $\|\cdot\|_{L^t}$.
	\item[$\bullet$] $C$ and $C_i$ denote positive constants that may change from line to line.
	\item[$\bullet$] If $A \subset \mathcal G$ is a measurable set, we denote by $|A|$ its Lebesgue measure.
\end{itemize}

\section{Functional setting and line profiles}\label{Sec2}
We introduce the metric graph, the function spaces, and the operator properties used below; a complete discussion of the definition and the features of metric graphs can be found in \cite{MR3385179,MR3013208,MR2042548} and the references therein.

\subsection{Graph spaces and the Dirac operator}\label{Sec2.1}
Throughout, a metric graph $\mathcal{G}=(V, E)$ is a connected multigraph (i.e., multiple edges and self-loops are allowed) with a finite number of edges and vertices. We indicate the set of edges of the graph with $E=\{e_j\}$ and the set of vertices of the graph with $V=\{\mathrm{v}_k\}$.
Each edge is a finite or half-infinite segment of line and the edges are glued together at their endpoints (the vertices of $\mathcal{G}$) according to the topology of the graph.

Unbounded edges are identified with $\mathbb{R}^{+}=[0,+\infty)$ and are called half-lines, while bounded edges are identified with closed and bounded intervals $I_e=\left[0, \ell_e\right]$, $\ell_e>0$. Each bounded edge is given an arbitrary orientation; each half-line is oriented away from its endpoint.

As a consequence, the graph $\mathcal{G}$ is a locally compact metric space, the metric given by the shortest distance along the edges. Clearly, since we assume a finite number of edges and vertices, $\mathcal{G}$ is compact if and only if it does not contain any half-line. A further important notion, introduced in \cite{MR3494248,MR3456809}, is the following.
\begin{Def}\label{Def2.1}
If $\mathcal{G}$ is a metric graph, we define its compact core $\mathcal{K}$ as the metric subgraph of $\mathcal{G}$ consisting of all its bounded edges. In addition, we denote by $\ell$ the measure of $\mathcal{K}$, namely
$$
\ell=\sum_{e \in \mathcal{K}} \ell_e.
$$
\end{Def}
\par
A function $u: \mathcal{G} \rightarrow \mathbb{C}$ can be regarded as a family of functions $(u_e)$, where $u_e: I_e \rightarrow \mathbb{C}$ is the restriction of $u$ to the edge (represented by) $I_e$. The usual $L^p$ spaces can be defined in the natural way, with norm
$$
\|u\|_{L^p(\mathcal{G} )}^p:=\sum_{e \in E}\left\|u_e\right\|_{L^p\left(I_e\right)}^{p}, \quad \text { if } p \in[1, \infty),
$$
and
$$
\|u\|_{L^{\infty}(\mathcal{G} )}:=\max _{e \in E}\left\|u_e\right\|_{L^{\infty}\left(I_e\right)}.
$$
$H^1(\mathcal{G})$ is the space of functions $u=\left(u_e\right)$ such that $u_e \in H^1\left(I_e\right)$ for every edge $e \in E$, with norm
$$
\|u\|_{H^1(\mathcal{G} )}^2=\left\|u^{\prime}\right\|_{L^2(\mathcal{G})}^2+\|u\|_{L^2(\mathcal{G})}^2.
$$
Consistently, a spinor $\psi=\left(u, v\right)^T: \mathcal{G} \rightarrow \mathbb{C}^2$ is a family of 2-spinors
$$
\psi_e=\binom{u_e}{v_e}: I_e \rightarrow \mathbb{C}^2,\quad \forall e \in E,
$$
and thus
$$
L^p\left(\mathcal{G}, \mathbb{C}^2\right):=\bigoplus_{e \in E} L^p\left(I_e\right) \otimes \mathbb{C}^2,
$$
endowed with the norm (using $|\psi_e|=(|u_e|^2+|v_e|^2)^{1/2}$)
\begin{equation*}
\|\psi\|_{L^p(\mathcal{G},\mathbb{C}^2)}^p:=\sum_{e \in E}\left\|\psi_e\right\|_{L^p\left(I_e\right)}^{p}, \quad \text { if } p \in[1, \infty),
\end{equation*}
and
\begin{equation*}
\|\psi\|_{L^{\infty}(\mathcal{G},\mathbb{C}^2)}:=\max _{e \in E}\left\|\psi_e\right\|_{L^{\infty}\left(I_e\right)}.
\end{equation*}
Moreover,
$$
H^1\left(\mathcal{G}, \mathbb{C}^2\right):=\bigoplus_{e \in E} H^1\left(I_e\right) \otimes \mathbb{C}^2,
$$
endowed with the norm
$$
\|\psi\|_{H^1\left(\mathcal{G}, \mathbb{C}^2\right)}^2:=\sum_{e \in E}\left\|\psi_e\right\|_{H^1\left(I_e\right)}^2.
$$
Equivalently, one can say that $L^p\left(\mathcal{G}, \mathbb{C}^2\right)$ is the space of the spinors such that $u, v \in L^p(\mathcal{G})$, with the following equivalent component norms:
$$
\begin{aligned}
& \|\psi\|_{L^p\left(\mathcal{G}, \mathbb{C}^2\right)}^p\asymp\left\|u\right\|_{L^p(\mathcal{G})}^p+\left\|v\right\|_{L ^ p(\mathcal{G})}^p \quad \text { if } p \in[1, \infty), \\
& \|\psi\|_{L^{\infty}\left(\mathcal{G}, \mathbb{C}^2\right)}\asymp\max \left\{\left\|u\right\|_{L^{\infty}(\mathcal{G})},\left\|v\right\|_{L^{\infty}(\mathcal{G})}\right\},
\end{aligned}
$$
where the equivalence constants depend only on $p$ (and the component formulas are exact for $p=2$). Moreover, $H^1\left(\mathcal{G}, \mathbb{C}^2\right)$ is the space of the spinors such that $u, v \in H^1(\mathcal{G})$, with
$$
\|\psi\|_{H^1\left(\mathcal{G}, \mathbb{C}^2\right)}^2:=\left\|u\right\|_{H^1(\mathcal{G})}^2+\left\|v\right\|_{H^1(\mathcal{G})}^2.
$$

Let
\begin{equation}\label{eq2.1}
\mathcal{D}:=-i c\frac{d}{dx}\sigma_1 +m c^2\sigma_3
\end{equation}
denote the Dirac operator. Then \eqref{eq1.1} is equivalent to the following equation
\begin{equation}\label{eq2.2}
\mathcal{D}\psi -\omega\psi=g(|\psi|)\psi , \quad \text { in } \mathcal{G}.
\end{equation}
The expression given by \eqref{eq2.1} on a metric graph is purely formal, since it does not clarify what happens at the vertices of the graph, given that the derivative $\frac{d}{d x}$ is well defined just in the interior of the edges.
As for the Laplacian in the Schr\"odinger case, the way to give a rigorous meaning to \eqref{eq2.1} is to find suitable self-adjoint realizations of the operator. For more details on self-adjoint extensions of the Dirac operator on metric graphs we refer the reader to \cite{MR1050469,MR2459887}.

\begin{Def}\label{Def2.2}
Let $\mathcal{G}$ be a metric graph and let $m, c>0$. We call the Dirac operator with Kirchhoff-type vertex conditions the operator $\mathcal{D}:dom(\mathcal D)\subset L^2\left(\mathcal{G}, \mathbb{C}^2\right) \rightarrow L^2\left(\mathcal{G}, \mathbb{C}^2\right)$ with action
\begin{equation}\label{eq2.3}
\mathcal{D}_{|I_e} \psi=\mathcal{D}_e \psi_e:=-i c \sigma_1 \psi_e^{\prime}+m c^2 \sigma_3\psi_e ,\quad \forall e \in E,
\end{equation}
and domain
\begin{equation}\label{eq2.4}
dom(\mathcal{D}):=\left\{\psi \in H^1\left(\mathcal{G}, \mathbb{C}^2\right): \psi \text { satisfies }\eqref{eq2.5} \text { and }\eqref{eq2.6}\right\},
\end{equation}
where
\begin{equation}\label{eq2.5}
u_e(\mathrm{v})=u_f(\mathrm{v}), \quad \forall e, f \succ \mathrm{v},\ \forall \mathrm{v} \in V,
\end{equation}
\begin{equation}\label{eq2.6}
\sum_{e\succ\mathrm{v}} v_e(\mathrm{v})_{ \pm}=0 , \quad \forall \mathrm{v} \in V,
\end{equation}
$"e \succ \mathrm{v}"$ meaning that the edge $e$ is incident at the vertex $\mathrm{v}$ and $v_e(\mathrm{v})_{ \pm}$ standing for $v_e(0)$ or $-v_e\left(\ell_e\right)$ according to whether $x_e$ is equal to $0$ or $\ell_e$ at $\mathrm{v}$.
In \eqref{eq2.5}--\eqref{eq2.6}, incidence is understood at the level of edge endpoints: a loop contributes both endpoint traces, with the corresponding signs in \eqref{eq2.6}. All upper-component endpoint traces at a vertex, including both endpoints of a loop, are required to agree.
\end{Def}

\begin{Rek}\label{Rek2.3}
The operator $\mathcal D$ depends on $m,c$. We suppress this dependence when no confusion can arise.
\end{Rek}
\par
The basic properties of the operator $\mathcal{D}$ with the above conditions are summarized in the following.
\begin{Lem}\cite{MR3934110}\label{Lem2.4}
The Dirac operator $\mathcal{D}$ introduced by Definition \ref{Def2.2} is self-adjoint on $L^2\left(\mathcal{G}, \mathbb{C}^2\right)$. For the noncompact graphs considered here, its spectrum is
\begin{equation}\label{eq2.7}
\sigma(\mathcal{D})=\left(-\infty,-m c^2\right] \cup\left[m c^2,+\infty\right).
\end{equation}
\end{Lem}

Define the space
\begin{equation}\label{eq2.8}
Y:=\left[L^2\left(\mathcal{G}, \mathbb{C}^2\right), dom(\mathcal{D})\right]_{\frac{1}{2}},
\end{equation}
namely, the complex interpolation space of order $\frac12$, with $dom(\mathcal D)$ equipped with its graph norm. The differential expression \eqref{eq2.3} shows that this graph norm is equivalent to the broken $H^1$ norm on $dom(\mathcal D)$. Interpolating the continuous inclusion $dom(\mathcal D)\hookrightarrow H^1(\mathcal G,\mathbb C^2)$ with the identity on $L^2$ gives a continuous embedding
\[
Y\hookrightarrow H^{\frac12}(\mathcal G,\mathbb C^2)
:=\bigoplus_{e\in E}H^{\frac12}(I_e)\otimes\mathbb C^2.
\]
We equip $Y$ with its interpolation norm. Consequently, the edgewise Sobolev embeddings imply
\begin{equation}\label{eq2.9}
Y \hookrightarrow L^p\left(\mathcal{G}, \mathbb{C}^2\right), \quad \forall p \in[2, \infty),
\end{equation}
and the restriction to every compact subgraph is compact in $L^p$, for $2\le p<\infty$. Here a compact subgraph may also contain finite initial segments of half-lines.

On the other hand, there holds (see \cite{MR3934110})
\begin{equation}\label{eq2.10}
dom\left(Q_{\mathcal{D}}\right)=Y,
\end{equation}
and hence the form domain inherits all the properties pointed out before, which are in fact crucial in the rest of the paper.

By the spectral theorem, $Y=dom(|\mathcal D|^{1/2})$. We use the convention $\langle z,w\rangle=z\cdot\overline w$ for the Hermitian product and regard $Y$ as a real Hilbert space when differentiating real-valued functionals. The associated real bilinear form is
\[
Q_{\mathcal D}(\psi,\varphi)
=\frac12\operatorname{Re}\bigl(\operatorname{sgn}(\mathcal D)|\mathcal D|^{1/2}\psi,
|\mathcal D|^{1/2}\varphi\bigr)_{L^2},
\qquad Q_{\mathcal D}(\psi)=Q_{\mathcal D}(\psi,\psi).
\]
When $\varphi\in dom(\mathcal D)$ this is
$\frac12\operatorname{Re}\int_{\mathcal G}\langle\psi,\mathcal D\varphi\rangle\,dx$.
The same integral notation below denotes the form extension whenever only form-domain regularity is available. The notation $\langle d\Phi(\psi),\varphi\rangle$ always denotes real duality.

\subsection{The action functional and basic estimates}\label{Sec2.2}
We first define $\Phi: Y \rightarrow \mathbb{R}$ by
\begin{equation}\label{eq2.11}
	\begin{aligned}
		\Phi(\psi)= &  \frac{1}{2}\operatorname{Re} \int_{\mathcal{G}}\left\langle\psi, {\mathcal{D}} \psi\right\rangle d x-\frac{\omega}{2} \int_{\mathcal{G}}|\psi|^2 d x- \int_{\mathcal{G}}G(| \psi|) d x \\
		=&   \frac{1}{2}\operatorname{Re} \int_{\mathcal{G}}\left\langle\psi,({\mathcal{D}} -\omega)\psi\right\rangle d x- \int_{\mathcal{G}}G(| \psi|) d x.
	\end{aligned}
\end{equation}
Assumptions $(g_1)$--$(g_4)$ and \eqref{eq2.9} imply that $\Phi$ is $C^1$ on this real Hilbert space. For the preliminary estimates in Section~\ref{Sec3.1} we set $c=1$, in which case \eqref{eq2.7} becomes
\begin{equation}\label{eq2.12}
	\sigma(\mathcal{D})=(-\infty,-m] \cup[m,+\infty).
\end{equation}
The critical points of $\Phi$ are precisely the solutions of \eqref{eq2.2} (see \cite{MR3934110}).

\begin{Lem}\label{Lem2.5}
A spinor $\psi$ is a solution of frequency $\omega$ of \eqref{eq2.2} if and only if it is a critical point of $\Phi$.
\end{Lem}
\begin{proof}
For $\psi\in Y$, assumptions $(g_1)$--$(g_4)$ and \eqref{eq2.9} give
$F(\psi):=g(|\psi|)\psi\in L^2(\mathcal G,\mathbb C^2)$: indeed,
$|F(\psi)|\le C(|\psi|+|\psi|^{p-1})$ and $\psi\in L^2\cap L^{2(p-1)}$.
If $\psi\in dom(\mathcal D)$ solves \eqref{eq2.2}, self-adjointness gives
\[
\langle d\Phi(\psi),\varphi\rangle
=\operatorname{Re}\int_{\mathcal G}\langle\psi,(\mathcal D-\omega)\varphi\rangle\,dx
-\operatorname{Re}\int_{\mathcal G}\langle F(\psi),\varphi\rangle\,dx=0
\]
for every $\varphi\in dom(\mathcal D)$. Density of $dom(\mathcal D)$ in $Y$ extends the identity to every $\varphi\in Y$.

Conversely, if $\psi\in Y$ is critical, then
\begin{equation}\label{eq2.13}
\operatorname{Re}\int_{\mathcal G}\langle\psi,\mathcal D\varphi\rangle\,dx
=\operatorname{Re}\int_{\mathcal G}\langle\omega\psi+F(\psi),\varphi\rangle\,dx,
\qquad \varphi\in dom(\mathcal D).
\end{equation}
Testing also with $i\varphi$ yields the corresponding complex identity. Since $\omega\psi+F(\psi)\in L^2$, the definition of the adjoint implies
$\psi\in dom(\mathcal D^*)$ and $\mathcal D^*\psi=\omega\psi+F(\psi)$.
By Lemma \ref{Lem2.4}, $\mathcal D^*=\mathcal D$, so $\psi\in dom(\mathcal D)$, including both vertex conditions. In particular, on each edge its components satisfy
\[
-icv_e'+(mc^2-\omega)u_e=g(|\psi_e|)u_e,
\qquad
-icu_e'-(mc^2+\omega)v_e=g(|\psi_e|)v_e.
\]
Thus $\psi$ solves \eqref{eq2.2}.
\end{proof}

Recall that according to Lemma \ref{Lem2.4} we can decompose the form domain $Y$ as the orthogonal sum of the positive and negative spectral subspaces for the operator $\mathcal{D}$, i.e.,
$$
Y=Y^{+} \oplus Y^{-}.
$$
As a consequence, every $\psi \in Y$ can be written as $\psi=P^{+} \psi+P^{-} \psi=: \psi^{+}+\psi^{-}$, where $P^{ \pm}$ are the orthogonal projectors onto $Y^{ \pm}$. In addition one can find an equivalent  norm for $Y$, i.e.,
\begin{equation}\label{eq2.14}
	\|\psi\|:=\|\sqrt{|\mathcal{D}|} \psi\|_{L^2} ,\quad \forall \psi \in Y.
\end{equation}
Combined with the above conclusion, we have the following lemma:

\begin{Lem}\label{Lem2.6}
$mc^2\|\psi\|_2^2 \leq  \|\psi\|^2$.
\end{Lem}
\begin{proof}
By the spectral theorem and \eqref{eq2.7},
\begin{align*}
\|\psi\|^2
&=(|\mathcal D|^{1/2}\psi,|\mathcal D|^{1/2}\psi)_{L^2}
=\int_{\sigma(\mathcal D)}|\lambda|\,d(E_\lambda\psi,\psi)_{L^2}
\ge mc^2\|\psi\|_2^2.
\end{align*}
Thus, we have
\begin{equation}\label{eq2.15}
mc^2\|\psi\|_2^2 \leq  \|\psi\|^2 , \quad \forall \psi \in Y.
\end{equation}
\end{proof}
\par
The spectral projections are contractions in the form norm. Together with \eqref{eq2.9}, this gives the estimate that will be used below:
\begin{equation}\label{eq2.16}
\|\psi^\pm\|_{L^q}\le C_q\|\psi^\pm\|\le C_q\|\psi\|,
\qquad \psi\in Y,\quad 2\le q<\infty.
\end{equation}
This estimate does not require an extension of $P^\pm$ to bounded operators on $L^q$.
Using the spectral theorem, the action functional \eqref{eq2.11} can be rewritten as follows:
\begin{equation}\label{eq2.17}
\Phi(\psi)=\frac{1}{2}\left(\left\|\psi^{+}\right\|^2-\left\|\psi^{-}\right\|^2\right)-\frac{\omega}{2} \int_{\mathcal{G}}|\psi|^2dx- \int_{\mathcal{G}}G(|\psi|) d x,
\end{equation}
which is convenient for establishing the linking geometry of $\Phi$ (see \cite{MR2431434} section II.8).
\par
Let $V$ be the space of the spinors
$$
\eta=\binom{\eta^1}{0} \text { where } \eta^1 \in \bigoplus_{e\in E}C_0^\infty(I_e^\circ,\mathbb{C}),
$$
which is clearly a subset of $Y$. Moreover, a simple computation shows that
\begin{equation}\label{eq2.18}
\begin{aligned}
\int_{\mathcal{G}}\left\langle\eta, \mathcal{D} \eta\right\rangle d x
&=\int_{\mathcal{G}}
\left\langle
\binom{\eta^1}{0},-i c\frac{d}{dx}\sigma_1 \binom{\eta^1}{0} +m c^2 \sigma_3\binom{\eta^1}{0}
\right\rangle d x \\
&=\int_{\mathcal{G}} \left\langle \binom{\eta^1}{0},\binom{mc^2\eta^1}{-ic(\eta^1)'}\right\rangle dx \\
&=mc^2\int_{\mathcal{G}}|\eta^1|^2 dx.
\end{aligned}
\end{equation}

\begin{Lem}\label{Lem2.7}
There exists $C>0$ such that
\[
G(s)\ge C s^{\theta} -C s^2,
\qquad \forall s\ge0.
\]
\end{Lem}
\begin{proof}
For $s>0$, set $H(s):=G(s)s^{-\theta}$. By $(g_4)$,
\[
H'(s)=\frac{s^2 g(s)-\theta G(s)}{s^{\theta+1}}\ge0,
\]
hence $H$ is nondecreasing on $(0,\infty)$. In particular, for $s\ge1$,
\[
G(s)=H(s)s^\theta \ge H(1)s^\theta.
\]
Let $C:=H(1)>0$. Then, for $s\ge1$,
\[
G(s)\ge C s^\theta \ge C s^\theta-Cs^2.
\]
If $0\le s\le1$, since $\theta>2$, one has $s^\theta\le s^2$, and therefore
\[
0\le G(s)\quad\Rightarrow\quad G(s)\ge C s^\theta-Cs^2.
\]
Combining the two cases yields the claim.
\end{proof}

We shall use the following inhomogeneous Gagliardo--Nirenberg inequality, which applies to the broken Sobolev space without continuity assumptions at vertices.
\begin{Lem}\label{Lem2.8}
Let $\mathcal G$ have finitely many edges. For every $2\le q\le\infty$ there exists $C_{q,\mathcal G}>0$ such that, for all $\psi\in H^1(\mathcal G,\mathbb C^2)$,
\begin{equation}\label{eq2.19}
\|\psi\|_{L^q(\mathcal G,\mathbb C^2)}
\le C_{q,\mathcal G}\|\psi\|_{L^2(\mathcal G,\mathbb C^2)}^{\frac12+\frac1q}
\|\psi\|_{H^1(\mathcal G,\mathbb C^2)}^{\frac12-\frac1q}.
\end{equation}
\end{Lem}
\begin{proof}
For a spinor $f\in H^1(0,\ell;\mathbb C^2)$, the fundamental theorem of calculus applied to $|f|^2$, starting from a point where $|f|^2$ does not exceed its average, gives
\[
\|f\|_\infty^2\le \ell^{-1}\|f\|_2^2+2\|f\|_2\|f'\|_2.
\]
On a half-line, integration toward infinity gives
$\|f\|_\infty^2\le2\|f\|_2\|f'\|_2$; this follows first for compactly supported smooth spinors and then by $H^1$ density.
Taking the maximum over the finitely many edges yields
\[
\|\psi\|_\infty^2\le C_{\mathcal G}\|\psi\|_2\|\psi\|_{H^1}.
\]
Interpolation with $L^2$, namely
$\|\psi\|_q\le\|\psi\|_2^{2/q}\|\psi\|_\infty^{1-2/q}$, proves \eqref{eq2.19}, including the endpoints.
\end{proof}

For $R>0$, let $\mathcal K_R$ contain all bounded edges and the initial $[0,R]$ segment of every half-line, with all vertices included. A family $\{z_j\}\subset L^2(\mathcal G,\mathbb C^2)$ is spatially tight if
\begin{equation}\label{eq2.20}
\lim_{R\to\infty}\sup_j\int_{\mathcal G\setminus\mathcal K_R}|z_j|^2\,dx=0.
\end{equation}
\subsection{Line profiles and escape levels}\label{Sec2.3}
Write $\delta=mc^2-\omega$, $K=mc^2+\omega$ and
\[
H(s)=\frac{2G(s)}{s^2}\quad(s>0),\qquad H(0)=0.
\]
By $(g_2)$, $(g_4)$ and $(g_5)$, $H$ is continuous, strictly increasing from zero to infinity, and
\begin{equation}\label{eq2.21}
H'(s)=\frac{4\hat G(s)}{s^3}>0.
\end{equation}
In particular, there is a unique $A>0$ with $H(A)=\delta$.

\begin{Pro}\label{Pro2.9}
For every $m,c>0$ and $|\omega|<mc^2$, the full-line equation has a profile $U_{c,\omega}=(f,ih)^T$ with $f>0$ even and $h$ odd. Every nonzero $H^1(\mathbb R,\mathbb C^2)$ solution is a translate and constant-phase multiple of this profile. Its components and first derivatives decay exponentially. Thus
\begin{equation}\label{eq2.22}
d^\infty_c(\omega):=\inf\{\Phi_{\mathbb R}(U):U\ne0,\ \Phi_{\mathbb R}'(U)=0\}
=\int_{\mathbb R}\hat G(|U_{c,\omega}|)\,dx>0.
\end{equation}
\end{Pro}
\begin{proof}
For real $f,h$, the equation for $(f,ih)^T$ is
\begin{equation}\label{eq2.23}
cf'=-(K+g(r))h,\qquad ch'=(g(r)-\delta)f,
\qquad r=(f^2+h^2)^{1/2}.
\end{equation}
It preserves $\delta f^2-Kh^2-2G(r)$. Start at $(f,h)=(A,0)$, so the conserved quantity is zero. Set $M=mc^2$ and $\alpha=\frac12\arccos(\omega/M)\in(0,\pi/2)$. In polar coordinates $(f,h)=r(\cos\vartheta,\sin\vartheta)$ the zero level satisfies
\begin{equation}\label{eq2.24}
H(r)=M\cos(2\vartheta)-\omega,\qquad
\vartheta'=\frac{g(r)-H(r)}c=\frac{2\hat G(r)}{cr^2}>0,
\qquad \frac{r'}r=-\frac Mc\sin(2\vartheta).
\end{equation}
For $x>0$ the orbit stays in $0<r\le A$, $0<\vartheta<\alpha$. It cannot reach $r=0$ in finite time by uniqueness for the locally Lipschitz system \eqref{eq2.23}. Boundedness extends it to all positive times. If the increasing angle had a limit smaller than $\alpha$, its derivative would stay bounded below by a positive constant. Hence $\vartheta\to\alpha$, $r\to0$, and
\[
\frac{r'}r\longrightarrow-\frac{\sqrt{m^2c^4-\omega^2}}c.
\]
This gives exponential decay of $r$ and, by \eqref{eq2.23}, of the first derivatives. Reflection with $f$ even and $h$ odd completes a full-line solution.

Conversely, let $(u,v)$ be a nonzero complex $H^1$ solution. The current $\operatorname{Re}(u\overline v)$ is constant and vanishes at infinity. At a point where $u\ne0$, choose a phase making $u$ real; then $v$ is purely imaginary. The corresponding real plane is invariant under the differential equation, so uniqueness gives $(u,v)=e^{i\gamma}(f,ih)$ with real $f,h$ on the whole line. The conserved quantity is zero because the solution tends to zero at both ends. Therefore $\delta f^2=Kh^2+2G(r)$, and $f$ never vanishes. Choose its sign positive. The angle belongs to $(-\alpha,\alpha)$ and is strictly increasing. Its limits at the two infinities are $-\alpha$ and $\alpha$, respectively. It passes through zero exactly once, where $r=A$. Uniqueness of the initial-value problem proves the classification. Finally, testing the equation with the solution gives $\Phi_{\mathbb R}(U)=\int\hat G(|U|)$, which is positive and finite.
\end{proof}

Fix $m>0$ and $\nu<0$. The scalar full-line equation
\begin{equation}\label{eq2.25}
-q''-\nu q=2mg(|q|)q
\end{equation}
has a unique positive even $H^1(\mathbb R)$ solution $q_0$, centered at its maximum. Its maximum $s_*$ is determined by $H(s_*)=-\nu/(2m)$. Define
\begin{equation}\label{eq2.26}
\begin{aligned}
e_{\mathbb R}(m,\nu)
&:=\int_{\mathbb R}\hat G(q_0)\,dx\\
&=2\int_0^{s_*}\frac{\hat G(s)}{\sqrt{-\nu s^2-4mG(s)}}\,ds>0.
\end{aligned}
\end{equation}

\begin{Pro}\label{Pro2.10}
Every nonzero complex $H^1$ solution of \eqref{eq2.25} is of the form $e^{i\gamma}q_0(\cdot-y)$. If \eqref{eq1.4}--\eqref{eq1.6} hold, the centered profiles $U_{c_n,\omega_n}=(f_n,ih_n)^T$ satisfy
\[
f_n\to q_0\text{ in }H^1(\mathbb R),\qquad
\|h_n\|_{H^1}=O(c_n^{-1}),\qquad
d^\infty_{c_n}(\omega_n)\to e_{\mathbb R}(m,\nu).
\]
\end{Pro}
\begin{proof}
The scalar first integral is
\begin{equation}\label{eq2.27}
|q'|^2=-\nu|q|^2-4mG(|q|),
\end{equation}
and $\operatorname{Im}(q'\overline q)$ is constant. These constants are fixed by decay at infinity; the equation first gives $q\in H^2$. A nonzero solution cannot vanish at a finite point, since \eqref{eq2.27} would also give $q'=0$ there. Its phase is therefore constant. For a positive real solution, the unique positive turning point is $s_*$, and the decreasing half-orbit is determined by \eqref{eq2.27}. The root at $s_*$ is simple by \eqref{eq2.21}; at zero the decay rate is $\sqrt{-\nu}$. This proves existence, uniqueness up to translation, and the quadrature \eqref{eq2.26}.

For the Dirac profiles write $\delta_n=mc_n^2-\omega_n$, $K_n=mc_n^2+\omega_n$, and $z_n=c_nh_n$. Their initial data are $(f_n,z_n)(0)=(H^{-1}(\delta_n),0)\to(s_*,0)$, and their equations become
\[
f_n'=-\left(\frac{K_n}{c_n^2}+\frac{g(r_n)}{c_n^2}\right)z_n,
\qquad z_n'=(g(r_n)-\delta_n)f_n,
\quad r_n=(f_n^2+z_n^2/c_n^2)^{1/2}.
\]
The zero first integral gives $|z_n|\le c_n\sqrt{\delta_n/K_n}\,f_n\le C f_n$, and $r_n\le H^{-1}(\delta_n)\le C$. Continuous dependence for this locally Lipschitz system gives convergence on every bounded interval to
$f'=-2mz$, $z'=(g(f)+\nu/(2m))f$, with initial data $(s_*,0)$. Thus $f=q_0$ and $z=-q_0'/(2m)$.
For $x>0$, \eqref{eq2.24} also gives
\[
-\frac{r_n'}{r_n}
=\sqrt{(\delta_n-H(r_n))\frac{K_n+H(r_n)}{c_n^2}}.
\]
Choose $r_0>0$ with $H(r_0)<-\nu/(8m)$. Local convergence and monotonicity of $r_n$ provide a common $R$ such that $r_n(R)<r_0$ for all sufficiently large $n$. The last square root is then bounded below by a positive constant for $x\ge R$. Together with the differential equations and $|z_n|\le C r_n$, this yields a common exponential majorant for $f_n,z_n$ and their first derivatives. The finitely many omitted indices cause no difficulty. Dominated convergence proves the stated Sobolev limits and, since $\hat G(s)=O(s^2)$ near zero, the convergence of the actions.
\end{proof}

\begin{Rek}\label{Rek2.11}
Putting the half-profile $U_{c,\omega}|_{[0,\infty)}$ on every edge of an $N$-star gives a nonzero solution for every gap frequency: the upper traces agree and all lower traces are zero. This observation does not require monotonicity of $g$. Under a monotonicity hypothesis the needed even--odd line profiles also follow from \cite[Proposition~3.1 and Lemma~3.2]{Yang2026}.
\end{Rek}
\section{Dirac bound states and ground states}\label{Sec3}
We establish linking geometry, boundedness of Cerami sequences, and compactness below the line action level. The last part identifies the ground-state level under $(g_6)$.

\subsection{Linking geometry and boundedness}\label{Sec3.1}
In this subsection we set $c=1$, fix $m>0$ and $\omega\in(-m,m)$, and work on a noncompact graph.

\begin{Lem}\label{Lem3.1}
We assume that $(g_1)-(g_5)$ are satisfied. Then there exist constants $\rho, r^*>0$ such that 
$$\kappa:=\inf \Phi\left(\partial B_\rho \cap Y^{+}\right) \geq r^*>0,$$
where $B_\rho=\{\psi \in Y:\|\psi\| \leq \rho\}$.
\end{Lem}
\begin{proof}
Since $c=1$ in these preliminary estimates and $\omega\in(-m,m)$, Lemma \ref{Lem2.6} yields
\begin{equation*}
\|\psi\|_{L^2}^2\le \frac1m\|\psi\|^2,
\qquad \forall\,\psi\in Y.
\end{equation*}
On the other hand, by $(g_2)$ and $(g_3)$, for every $\varepsilon>0$ there exists $C_\varepsilon>0$ such that
\begin{equation*}
G(s)\le \varepsilon s^2+C_\varepsilon s^p,
\qquad \forall s\ge0.
\end{equation*}
Hence, for every $\psi\in Y^+$,
\begin{align*}
\Phi(\psi)
&=\frac12\|\psi\|^2-\frac\omega2\|\psi\|_{L^2}^2-\int_{\mathcal G}G(|\psi|)\,dx \\
&\ge \frac12\|\psi\|^2-\frac\omega2\|\psi\|_{L^2}^2-\varepsilon\|\psi\|_{L^2}^2-C_\varepsilon\|\psi\|_{L^p}^p \\
&\ge \left(\frac12-\frac{|\omega|}{2m}-\frac\varepsilon m\right)\|\psi\|^2-C\|\psi\|^p,
\end{align*}
where we also used the continuous embedding $Y\hookrightarrow L^p(\mathcal G,\mathbb C^2)$ from \eqref{eq2.9}. Choose $\varepsilon>0$ so small that
\begin{equation*}
\alpha:=\frac12-\frac{|\omega|}{2m}-\frac\varepsilon m>0.
\end{equation*}
Then
\begin{equation*}
\Phi(\psi)\ge \alpha\|\psi\|^2-C\|\psi\|^p,
\qquad \forall\,\psi\in Y^+.
\end{equation*}
Since $p>2$, the right-hand side is positive on $\partial B_\rho\cap Y^+$ for $\rho>0$ sufficiently small. Therefore there exist $\rho,r^*>0$ such that
\begin{equation*}
\inf_{\partial B_\rho\cap Y^+}\Phi\ge r^*>0.
\end{equation*}
This proves the claim.
\end{proof}

\par
Choose $e^+\in Y^+\setminus\{0\}$ with $\|e^+\|=1$, put $V_0:=\mathbb Re^+$, and define $\hat Y:=Y^-\oplus\mathbb R^+e^+$, where $\mathbb R^+:=[0,\infty)$.

\begin{Lem}\label{Lem3.2}
We assume that $(g_1)-(g_5)$ are satisfied. Then $\sup \Phi(\hat{Y})<\infty$, and there is a constant $R>0$ such that $\sup \Phi\left(\hat{Y} \backslash B_R\right) \leq 0$, where $ B_R = \{ \psi  \in  \hat{Y} : \|\psi\| \leq R\}$.
\end{Lem}
\begin{proof}
We first show that
\begin{equation}\label{eq3.1}
\Phi(\psi)\longrightarrow -\infty
\qquad \text{as }\|\psi\|\to\infty,\ \psi\in\hat Y.
\end{equation}
Once \eqref{eq3.1} is proved, both assertions of the lemma follow at once.

Assume by contradiction that \eqref{eq3.1} fails. Then there exist a sequence $\{\psi_j\}\subset\hat Y$ and a constant $M>0$ such that
\begin{equation}\label{eq3.2}
\|\psi_j\|\to\infty,
\qquad
\Phi(\psi_j)\ge -M,
\qquad j\to\infty.
\end{equation}
Since $\psi_j\in\hat Y$, we can write
\[
\psi_j=\varphi_j+t_je^+,
\qquad
\varphi_j\in Y^-,\quad t_j\ge0.
\]
Define the normalized sequence
\[
z_j:=\frac{\psi_j}{\|\psi_j\|}=y_j+s_je^+,
\qquad
y_j:=\frac{\varphi_j}{\|\psi_j\|}\in Y^-,
\qquad
s_j:=\frac{t_j}{\|\psi_j\|}\ge0.
\]
Clearly $\|z_j\|=1$ for every $j$. Since $Y^-$ is weakly closed and $V_0=\mathbb Re^+$ is finite-dimensional, up to a subsequence there exist $y\in Y^-$ and $s\ge0$ such that
\begin{equation}\label{eq3.3}
y_j\rightharpoonup y \text{ in }Y^-,
\qquad
s_j\to s,
\qquad
z_j\rightharpoonup z:=y+se^+ \text{ in }Y.
\end{equation}
By the local compactness of the embedding $Y\hookrightarrow L^\theta(\mathcal K,\mathbb C^2)$ for every compact subgraph $\mathcal K$, we also have
\begin{equation}\label{eq3.4}
z_j\to z \text{ in }L^\theta(\mathcal K,\mathbb C^2)
\quad\text{for every compact subgraph }\mathcal K\subset\mathcal G.
\end{equation}

We now distinguish two cases.

\smallskip
\noindent
\emph{Case 1: $s=0$.}
In this case $z_j^+=s_je^+\to0$ strongly in $Y$, because $V_0$ is one-dimensional. Since $z_j=z_j^++z_j^-$ and $\|z_j\|=1$, it follows that $\|z_j^-\|\to1$. Dividing the identity defining $\Phi(\psi_j)$ by $\|\psi_j\|^2$ and using the positivity of $G$, we obtain
\begin{align*}
\frac{\Phi(\psi_j)}{\|\psi_j\|^2}
&=\frac12\|z_j^+\|^2-\frac12\|z_j^-\|^2-\frac\omega2\|z_j\|_{L^2}^2
 -\frac1{\|\psi_j\|^2}\int_{\mathcal G}G(|\psi_j|)\,dx \\
&\le \frac12\|z_j^+\|^2-\frac12\|z_j^-\|^2
+\frac{|\omega|}{2}\|z_j\|_{L^2}^2.
\end{align*}
By Lemma \ref{Lem2.6} with $c=1$, applied to $z_j^+$ and $z_j^-$,
\[
\|z_j^\pm\|_{L^2}^2\le \frac1m\|z_j^\pm\|^2.
\]
Hence
\[
\frac{\Phi(\psi_j)}{\|\psi_j\|^2}
\le
\left(\frac12+\frac{|\omega|}{2m}\right)\|z_j^+\|^2
-\left(\frac12-\frac{|\omega|}{2m}\right)\|z_j^-\|^2.
\]
Since $|\omega|<m$, $\|z_j^+\|\to0$ and $\|z_j^-\|\to1$, we get
\[
\limsup_{j\to\infty}\frac{\Phi(\psi_j)}{\|\psi_j\|^2}
\le
-\left(\frac12-\frac{|\omega|}{2m}\right)<0.
\]
This contradicts \eqref{eq3.2}, because \eqref{eq3.2} gives
\[
\frac{\Phi(\psi_j)}{\|\psi_j\|^2}\ge -\frac{M}{\|\psi_j\|^2}\to0.
\]

\smallskip
\noindent
\emph{Case 2: $s>0$.}
Since $y\in Y^-$, $se^+\in Y^+$, and $Y^-\cap Y^+=\{0\}$, the condition $s>0$ implies
$z=y+se^+\neq0$. Hence there exist a compact subgraph $\mathcal K\subset\mathcal G$, a measurable set $\Omega\subset\mathcal K$ with $|\Omega|>0$, and a number $\delta>0$ such that
\[
|z(x)|\ge \delta
\qquad \text{for a.e. }x\in\Omega.
\]
Using \eqref{eq3.4}, for $j$ large enough we have
\[
|z_j(x)|\ge \frac\delta2
\qquad \text{for a.e. }x\in\Omega_j,
\]
where $\Omega_j\subset\Omega$ satisfies $|\Omega_j|\ge |\Omega|/2$. By Lemma \ref{Lem2.7}, there exists $C>0$ such that
\[
G(s)\ge Cs^\theta-Cs^2,
\qquad \forall s\ge0.
\]
Therefore,
\begin{align*}
\frac{\Phi(\psi_j)}{\|\psi_j\|^2}
&=\frac12\|z_j^+\|^2-\frac12\|z_j^-\|^2-\frac\omega2\|z_j\|_{L^2}^2
-\frac1{\|\psi_j\|^2}\int_{\mathcal G}G(|\psi_j|)\,dx \\
&\le \frac12\|z_j^+\|^2-\frac12\|z_j^-\|^2
+\frac{|\omega|}{2}\|z_j\|_{L^2}^2
-\frac1{\|\psi_j\|^2}\int_{\Omega_j}G(|\psi_j|)\,dx \\
&\le C
-C\|\psi_j\|^{\theta-2}\int_{\Omega_j}|z_j|^\theta\,dx
+C\int_{\Omega_j}|z_j|^2\,dx.
\end{align*}
Here we used that $\{z_j\}$ is bounded in $Y$, hence also in $L^2(\mathcal G,\mathbb C^2)$. Since $\theta>2$ and
\[
\int_{\Omega_j}|z_j|^\theta\,dx\ge \left(\frac\delta2\right)^\theta\frac{|\Omega|}{2}>0,
\]
the right-hand side tends to $-\infty$, again contradicting \eqref{eq3.2}.

The contradiction proves \eqref{eq3.1}. Hence there exists $R>0$ such that
\[
\sup_{\hat Y\setminus B_R}\Phi\le0.
\]
On the other hand, if $\psi\in \hat Y\cap B_R$, then $\|\psi\|\le R$, and by \eqref{eq2.17}, Lemma \ref{Lem2.6}, $(g_3)$, and the embedding \eqref{eq2.9},
\begin{align*}
\Phi(\psi)
&\le \frac12\|\psi^+\|^2+\frac{|\omega|}{2}\|\psi\|_{L^2}^2+\int_{\mathcal G}|G(|\psi|)|\,dx \\
&\le C\bigl(\|\psi\|^2+\|\psi\|_{L^p}^p\bigr)
\le C_R.
\end{align*}
Therefore $\sup_{\hat Y\cap B_R}\Phi<\infty$, and combining this with the previous estimate yields $\sup_{\hat Y}\Phi<\infty$. The proof is complete.
\end{proof}

\begin{Lem}\label{Lem3.3}
For every $N\in\mathbb N$ and every $N$-dimensional real subspace $Z_N\subset Y^+$, there exists $R>0$ such that
\begin{equation}\label{eq3.5}
\Phi(\psi)\le0,\qquad \psi\in\partial B_R,
\end{equation}
where
\[
B_R=\{\psi\in Y^-\oplus Z_N:\ \|\psi^-\|\le R,\ \|\psi^+\|\le R\},
\]
and the boundary is taken relative to $Y^-\oplus Z_N$.
\end{Lem}
\begin{proof}
The positive spectral subspace is infinite-dimensional by Lemma \ref{Lem2.4}, so such subspaces exist. We prove the stronger assertion that $\Phi(\psi)\to-\infty$ as $\|\psi\|\to\infty$ in $Y^-\oplus Z_N$.
Otherwise there exist $\psi_j\in Y^-\oplus Z_N$ with $t_j:=\|\psi_j\|\to\infty$ and $\Phi(\psi_j)\ge-M$ for some $M>0$. Write
\[
z_j:=\psi_j/t_j=y_j+\xi_j,\qquad y_j\in Y^-,\quad\xi_j\in Z_N,\quad\|z_j\|=1.
\]
After taking a subsequence, $y_j\rightharpoonup y$ in $Y^-$ and $\xi_j\to\xi$ strongly in $Z_N$. Thus $z_j\rightharpoonup z:=y+\xi$ in $Y$, with local strong convergence in every finite $L^q$ space. Put
\[
a_\omega:=\frac12+\frac{|\omega|}{2m},\qquad
b_\omega:=\frac12-\frac{|\omega|}{2m}>0.
\]
By \eqref{eq2.17}, $L^2$ orthogonality and Lemma \ref{Lem2.6},
\begin{equation}\label{eq3.6}
\frac{\Phi(\psi_j)}{t_j^2}
\le a_\omega\|\xi_j\|^2-b_\omega\|y_j\|^2
-\frac1{t_j^2}\int_{\mathcal G}G(t_j|z_j|)\,dx.
\end{equation}
If $\xi=0$, then $\|\xi_j\|\to0$ and $\|y_j\|\to1$; positivity of $G$ makes the upper limit in \eqref{eq3.6} at most $-b_\omega<0$, contradicting $\Phi(\psi_j)/t_j^2\ge-M/t_j^2$.
If $\xi\ne0$, then $z=y+\xi\ne0$. There is a compact subgraph $K$ with $\int_K|z|^\theta\,dx>0$. Local strong convergence and Lemma \ref{Lem2.7} imply
\[
\frac1{t_j^2}\int_{\mathcal G}G(t_j|z_j|)\,dx
\ge Ct_j^{\theta-2}\int_K|z_j|^\theta\,dx-C\int_K|z_j|^2\,dx
\longrightarrow+\infty.
\]
This again contradicts \eqref{eq3.6} and $\Phi(\psi_j)\ge-M$. The stronger assertion follows. Every point on $\partial B_R$ has norm at least $R$, which proves \eqref{eq3.5} for sufficiently large $R$.
\end{proof}

The following boundedness estimate is used in the compactness argument.

\begin{Lem}\label{Lem3.4}
Any $(C)_d$ sequence at a finite level $d$ is bounded in $Y$.
\end{Lem}
\begin{proof}
Let $\{\psi_j\}\subset Y$ be a $(C)_d$ sequence with $d\in\mathbb R$. Since the Cerami condition gives
\[
\langle d\Phi(\psi_j),\psi_j\rangle=o(1),
\]
we have
\begin{equation}\label{eq3.7}
\Phi(\psi_j)=d+o(1),
\qquad
\Phi(\psi_j)-\frac12\langle d\Phi(\psi_j),\psi_j\rangle \le C.
\end{equation}
Consequently,
\begin{equation}\label{eq3.8}
\int_{\mathcal G}\hat G(|\psi_j|)\,dx
=\Phi(\psi_j)-\frac12\langle d\Phi(\psi_j),\psi_j\rangle
\le C,
\end{equation}
where $\hat G(s):=\frac12 g(s)s^2-G(s)$.

Assume by contradiction that $\|\psi_j\|\to\infty$. Set
\[
\varphi_j:=\frac{\psi_j}{\|\psi_j\|}.
\]
Then $\{\varphi_j\}$ is bounded in $Y$, and therefore, by \eqref{eq2.9}, it is bounded in $L^s(\mathcal G,\mathbb C^2)$ for every $s\in[2,\infty)$.

Fix $\rho>0$ and define
\[
I_j(\rho):=\{x\in\mathcal G:\ |\psi_j(x)|\ge \rho\},
\qquad
I_j^c(\rho):=\mathcal G\setminus I_j(\rho).
\]
We claim that, for every $\rho>0$, there exists $a_\rho>0$ such that
\begin{equation}\label{eq3.9}
\hat G(s)\ge a_\rho s^p,
\qquad \forall s\ge \rho.
\end{equation}
Indeed, if $\rho\ge R$, then \eqref{eq3.9} follows directly from $(g_5)$ with $a_\rho=c_1$. If $0<\rho<R$, since $\hat G(s)>0$ for every $s>0$, the function $s\mapsto \hat G(s)s^{-p}$ is continuous and positive on the compact interval $[\rho,R]$, hence it admits a positive minimum there; combining this minimum with the estimate for $s\ge R$ again gives \eqref{eq3.9}.

Using \eqref{eq3.8} and \eqref{eq3.9}, we obtain
\begin{align}
\int_{I_j(\rho)}|\varphi_j|^p\,dx
&=\frac1{\|\psi_j\|^p}\int_{I_j(\rho)}|\psi_j|^p\,dx \nonumber\\
&\le \frac1{a_\rho\|\psi_j\|^p}\int_{I_j(\rho)}\hat G(|\psi_j|)\,dx \nonumber\\
&\le \frac{C}{a_\rho\|\psi_j\|^p}\longrightarrow 0.
\label{eq3.10}
\end{align}

Now let $\varepsilon>0$. By $(g_2)$ there exists $\rho_\varepsilon>0$ such that
\[
|g(t)|\le \varepsilon,
\qquad \forall\,0\le t\le \rho_\varepsilon.
\]
Evaluating $d\Phi(\psi_j)$ on $\psi_j^+-\psi_j^-$, we get
\begin{align*}
\langle d\Phi(\psi_j),\psi_j^+-\psi_j^-\rangle
=&\|\psi_j\|^2-\omega\bigl(\|\psi_j^+\|_{L^2}^2-\|\psi_j^-\|_{L^2}^2\bigr) \\
&\hspace{3em}-\operatorname{Re}\int_{\mathcal G}g(|\psi_j|)\psi_j\cdot\overline{(\psi_j^+-\psi_j^-)}\,dx.
\end{align*}
Since
\[
-\omega\bigl(\|\psi_j^+\|_{L^2}^2-\|\psi_j^-\|_{L^2}^2\bigr)
\ge
-|\omega|\|\psi_j\|_{L^2}^2,
\]
Lemma \ref{Lem2.6} with $c=1$ yields
\[
\|\psi_j\|^2-\omega\bigl(\|\psi_j^+\|_{L^2}^2-\|\psi_j^-\|_{L^2}^2\bigr)
\ge
\left(1-\frac{|\omega|}{m}\right)\|\psi_j\|^2.
\]
Using also the Cerami condition and splitting the nonlinear term on $I_j(\rho_\varepsilon)$ and $I_j^c(\rho_\varepsilon)$, we obtain
\begin{align*}
\left(1-\frac{|\omega|}{m}\right)\|\psi_j\|^2
\leq &\; o(1)
+\int_{I_j^c(\rho_\varepsilon)}|g(|\psi_j|)|\,|\psi_j|\,|\psi_j^+-\psi_j^-|\,dx \\
&\; +\int_{I_j(\rho_\varepsilon)}|g(|\psi_j|)|\,|\psi_j|\,|\psi_j^+-\psi_j^-|\,dx.
\end{align*}
After division by $\|\psi_j\|^2$ and using $\varphi_j=\psi_j/\|\psi_j\|$, this becomes
\begin{align*}
1-\frac{|\omega|}{m}
\leq &\; o(1)
+\int_{I_j^c(\rho_\varepsilon)}|g(|\psi_j|)|\,|\varphi_j|\,|\varphi_j^+-\varphi_j^-|\,dx \\
&\; +\frac1{\|\psi_j\|^2}\int_{I_j(\rho_\varepsilon)}|g(|\psi_j|)|\,|\psi_j|\,|\psi_j^+-\psi_j^-|\,dx.
\end{align*}
On $I_j^c(\rho_\varepsilon)$ we have $|g(|\psi_j|)|\le \varepsilon$, hence
\[
\int_{I_j^c(\rho_\varepsilon)}|g(|\psi_j|)|\,|\varphi_j|\,|\varphi_j^+-\varphi_j^-|\,dx
\le \varepsilon\int_{\mathcal G}|\varphi_j|\,|\varphi_j^+-\varphi_j^-|\,dx
\le C\varepsilon.
\]
On the other hand, by $(g_3)$ and since $|\psi_j|\ge \rho_\varepsilon$ on $I_j(\rho_\varepsilon)$, there exists $C_{\rho_\varepsilon}>0$ such that
\[
|g(t)|\,t\le C_{\rho_\varepsilon} t^{p-1},
\qquad \forall\,t\ge \rho_\varepsilon.
\]
Therefore
\begin{align*}
&\frac1{\|\psi_j\|^2}\int_{I_j(\rho_\varepsilon)}|g(|\psi_j|)|\,|\psi_j|\,|\psi_j^+-\psi_j^-|\,dx \\
&\qquad\le C_{\rho_\varepsilon}\|\psi_j\|^{p-2}\int_{I_j(\rho_\varepsilon)}|\varphi_j|^{p-1}\,|\varphi_j^+-\varphi_j^-|\,dx \\
&\qquad\le C_{\rho_\varepsilon}\|\psi_j\|^{p-2}\left(\int_{I_j(\rho_\varepsilon)}|\varphi_j|^p\,dx\right)^{\frac{p-1}{p}}
\|\varphi_j^+-\varphi_j^-\|_{L^p(\mathcal G)} \\
&\qquad\le \frac{C_{\rho_\varepsilon}}{\|\psi_j\|}\longrightarrow 0
\end{align*}
by the quantitative bound in \eqref{eq3.10}. Here $\|\varphi_j^+-\varphi_j^-\|=\|\varphi_j\|=1$, so \eqref{eq2.9} gives the required uniform $L^p$ bound without any $L^p$ boundedness assumption on the spectral projections.

We conclude that
\[
1-\frac{|\omega|}{m}\le C\varepsilon.
\]
Since $|\omega|<m$, choosing $\varepsilon>0$ sufficiently small yields a contradiction. Hence $\{\psi_j\}$ is bounded in $Y$.
\end{proof}

\subsection{Linking and compactness below the line level}\label{Sec3.2}
We now fix $m,c>0$ and $|\omega|<mc^2$. The estimates above apply after dividing the equation and action by $c$: the mass, frequency and nonlinearity become $mc$, $\omega/c$ and $g/c$. The form domain is unchanged and its norm is multiplied by $c^{-1/2}$. Put $M=mc^2$, $d_\infty=d^\infty_c(\omega)$ and denote the full-line action by $\Phi_\infty$.

\begin{Lem}\label{Lem3.5}
For every $e\in Y^+\setminus\{0\}$, there is a Cerami sequence for $\Phi$ at a level $d$ satisfying
\[
0<\kappa\le d\le\Gamma_{c,\omega}(e),
\qquad \Gamma_{c,\omega}(e)=\sup_{t\ge0,\,\eta\in Y^-}\Phi(te+\eta),
\]
where $\kappa>0$ is independent of $e$ for fixed $m,c,\omega$ and $\mathcal G$.
\end{Lem}
\begin{proof}
Use the equivalent Hilbert norm $\|z\|_\omega=\||\mathcal D_c-\omega|^{1/2}z\|_2$. The positive and negative spectral subspaces coincide with those of $\mathcal D_c$, because $|\omega|<mc^2$. The form space $Y$ is separable. Regarding it as a real Hilbert space, write
\[
\Phi(z)=\tfrac12\|z^+\|_\omega^2-\tfrac12\|z^-\|_\omega^2-\Psi(z),
\qquad \Psi(z)=\int_{\mathcal G}G(|z|)\,dx.
\]
The functional $\Psi$ is $C^1$ and nonnegative. It is weakly sequentially lower semicontinuous: from a sequence realizing the lower limit, local compactness supplies an almost-everywhere convergent subsequence, and Fatou's lemma applies to $G\ge0$.
If $z_j\rightharpoonup z$ in $Y$, then $F(z_j)=g(|z_j|)z_j$ is bounded in $L^2$ and converges locally strongly in $L^2$ to $F(z)$. Thus $F(z_j)\rightharpoonup F(z)$ in $L^2$, by approximation of a fixed $L^2$ test with compactly supported tests. Consequently $\Psi'(z_j)[h]\to\Psi'(z)[h]$ for every $h\in Y$, and $\Psi'$ is weakly sequentially continuous.

The small-sphere estimate of Lemma~\ref{Lem3.1}, in this equivalent norm, gives $r,\kappa>0$ with $\Phi\ge\kappa$ on $N=\{z\in Y^+:\|z\|_\omega=r\}$. The normalized finite-dimensional argument of Lemma~\ref{Lem3.2} gives $\Phi(z)\to-\infty$ as $\|z\|_\omega\to\infty$ in $Y^-\oplus\mathbb R^+e$. Choose $R>r$ such that the relative boundary of
\[
Q=\{te+\eta:t\ge0,\ \eta\in Y^-,\ \|te+\eta\|_\omega\le R\}
\]
has nonpositive action. On its base $t=0$, this also follows directly from $\Psi\ge0$. All assumptions of the Cerami linking theorem \cite[Theorem~2.1]{LiSzulkin2002} are now verified. It yields a Cerami sequence at a level at least $\kappa$. The identity deformation is admissible in its minimax class, so the level is at most $\sup_Q\Phi\le\Gamma_{c,\omega}(e)$. Equivalence of the norms preserves the Cerami condition.
\end{proof}
\begin{Lem}\label{Lem3.6}
Let $\{w_j\}$ be bounded in $Y$ and satisfy $w_j\to0$ in $L^2(\mathcal K_R)$ for every $R>0$. On each half-line $e$, let $I_k=[k,k+1]$, $k\in\mathbb N_0$. If
\begin{equation}\label{eq3.11}
\sup_{e\text{ half-line},\ k\in\mathbb N_0}
\int_{I_k}|w_{j,e}|^2\,dx\longrightarrow0,
\end{equation}
then $w_j\to0$ in $L^q(\mathcal G,\mathbb C^2)$ for every $2<q<\infty$.
\end{Lem}
\begin{proof}
The form-domain norm continuously controls the broken $H^{1/2}$ norm. On every interval $I_k$, the $H^{1/2}(I_k)\hookrightarrow L^4(I_k)$ constant is independent of $k$. Thus
\begin{align*}
\int_0^\infty|w_{j,e}|^3\,dx
&\le\sum_{k=0}^\infty
\|w_{j,e}\|_{L^2(I_k)}\|w_{j,e}\|_{L^4(I_k)}^2\\
&\le C\left(\sup_k\int_{I_k}|w_{j,e}|^2\,dx\right)^{1/2}
\|w_{j,e}\|_{H^{1/2}(\mathbb R_+)}^2.
\end{align*}
Here the sum of the squared local $H^{1/2}$ norms is bounded by the global squared norm: using the Gagliardo norm, this follows by retaining only pairs of points in the same $I_k$. Sum over the finitely many half-lines. On the bounded edges, local compactness gives $L^3$ convergence to zero. Consequently $w_j\to0$ in $L^3(\mathcal G)$. Interpolation with boundedness in $L^2$ and in arbitrarily large finite $L^r$ proves the assertion.
\end{proof}

\begin{Lem}\label{Lem3.7}
Suppose $z_j,z\in Y$, $\{z_j\}$ is bounded in $Y$, and $z_j-z\to0$ in $L^3(\mathcal G)$. Then
\[
F(z_j)-F(z)\longrightarrow0\quad\hbox{in }L^2(\mathcal G,\mathbb C^2),
\qquad F(z)=g(|z|)z.
\]
\end{Lem}
\begin{proof}
The $L^2$ norms of $z_j,z$ are bounded. Moreover, the growth of $F$ and the embeddings of $Y$ imply uniform $L^4$ bounds for $F(z_j)$ and $F(z)$; it suffices here to use $Y\hookrightarrow L^4\cap L^{4p-4}$.

Fix $\varepsilon>0$ and choose $\delta>0$ with $|g(s)|\le\varepsilon$ for $0\le s\le\delta$. For $R>\delta$, split the graph according as
\[
B_{0,j}=\{\max(|z_j|,|z|)\le\delta\},\quad
B_{1,j}=\{\delta<\max(|z_j|,|z|)\le R\},\quad
B_{2,j}=\{\max(|z_j|,|z|)>R\}.
\]
On $B_{0,j}$ the $L^2$ norm of the difference is at most $C\varepsilon$. On $B_{1,j}$, the map $F:\mathbb R^4\to\mathbb R^4$ is Lipschitz on the ball of radius $R$, by $(g_1)$, and $|B_{1,j}|\le C\delta^{-2}$. Therefore
\[
\|F(z_j)-F(z)\|_{L^2(B_{1,j})}
\le L_R|B_{1,j}|^{1/6}\|z_j-z\|_3\longrightarrow0.
\]
Finally $|B_{2,j}|\le CR^{-2}$, so the uniform $L^4$ bounds give
\[
\|F(z_j)-F(z)\|_{L^2(B_{2,j})}\le CR^{-1/2}.
\]
First let $j\to\infty$, then $R\to\infty$, and finally $\varepsilon\to0$.
\end{proof}

\begin{Pro}\label{Pro3.8}
Assume $(g_1)$--$(g_5)$ and let $\{\psi_j\}\subset Y$ be a Cerami sequence for $\Phi$ at a finite level $d>0$. After extraction, there is a critical point $\psi\in Y$ such that $\psi_j\rightharpoonup\psi$ in $Y$, with
\begin{equation}\label{eq3.12}
0\le\Phi(\psi)=\int_{\mathcal G}\hat G(|\psi|)\,dx\le d.
\end{equation}
Either $\psi_j\to\psi$ strongly in $Y$, or, after a further extraction, there exist a half-line $e$, numbers $t_j\to\infty$, and a nonzero full-line critical point $\varphi$ such that
\[
\psi_{j,e}(t_j+\cdot)\longrightarrow\varphi
\quad\hbox{in }L^q_{\mathrm{loc}}(\mathbb R,\mathbb C^2)
\quad(2\le q<\infty),
\]
and
\begin{equation}\label{eq3.13}
d\ge\Phi(\psi)+\Phi_\infty(\varphi)\ge\Phi(\psi)+d_\infty.
\end{equation}
The translated functions are considered on $(-t_j,\infty)$, so their restrictions to every fixed compact interval are defined for all sufficiently large $j$.
\end{Pro}
\begin{proof}
Boundedness follows from Lemma~\ref{Lem3.4}; for general fixed $c$, apply that lemma to $\mathcal D_c/c$, mass $mc$, frequency $\omega/c$, and nonlinearity $g/c$. Extract a weak limit $\psi\in Y$. Local compactness gives strong convergence on every compact subgraph in all finite $L^q$, $q\ge2$, and almost-everywhere convergence.

The sequence $F(\psi_j)$ is bounded in $L^2$, because $Y\hookrightarrow L^2\cap L^{2p-2}$. Its local strong $L^2$ limit is $F(\psi)$, so $F(\psi_j)\rightharpoonup F(\psi)$ in $L^2$. Passing to the limit against an arbitrary fixed test in $Y$ gives $\Phi'(\psi)=0$. In particular, Lemma~\ref{Lem2.5} gives the graph vertex conditions and operator-domain regularity. The Cerami identity gives
\begin{equation}\label{eq3.14}
\int_{\mathcal G}\hat G(|\psi_j|)\,dx
=\Phi(\psi_j)-\frac12\langle\Phi'(\psi_j),\psi_j\rangle
\longrightarrow d.
\end{equation}
Since $\hat G\ge0$, Fatou's lemma and the critical-point identity prove \eqref{eq3.12}.

Put $w_j=\psi_j-\psi$. If \eqref{eq3.11} holds for $w_j$, Lemmas~\ref{Lem3.6} and \ref{Lem3.7} give $F(\psi_j)-F(\psi)\to0$ in $L^2$. Let $S=\operatorname{sgn}(\mathcal D_c)$. Since $\|Sw_j\|_Y=\|w_j\|_Y$, subtraction of the two derivative identities yields
\[
\left(1-\frac{|\omega|}{M}\right)\|w_j\|_Y^2
\le\left(\|\Phi'(\psi_j)\|_{Y^*}
+M^{-1/2}\|F(\psi_j)-F(\psi)\|_2\right)\|w_j\|_Y.
\]
Consequently $w_j\to0$ in $Y$.

If strong convergence fails, pass to a subsequence on which $\|w_j\|_Y$ stays bounded away from zero. By the preceding argument, \eqref{eq3.11} cannot hold even after further extraction. Thus, after extraction, there exist $\eta>0$, a single half-line $e$, and integers $k_j\ge0$ such that
\[
\int_{k_j}^{k_j+1}|w_{j,e}|^2\,dx\ge\eta.
\]
Local convergence forces $k_j\to\infty$. Since $\psi_e\in L^2(\mathbb R_+)$, its $L^2$ norm on these translated intervals tends to zero. By the triangle inequality, the same intervals carry an $L^2$ norm of $\psi_{j,e}$ bounded below by $\sqrt\eta/2$ for large $j$. Set $t_j=k_j$.

To obtain global weak compactness for the translated profiles, extend $\psi_{j,e}$ evenly from $\mathbb R_+$ to $\mathbb R$ and then translate by $t_j$. Even reflection is bounded from $H^{1/2}(\mathbb R_+)$ to $H^{1/2}(\mathbb R)$: in the Gagliardo seminorm, the cross-term has denominator $(x+y)^2\ge|x-y|^2$, and is bounded by the same-side seminorm. Translation is an isometry. Hence these extensions are bounded in $H^{1/2}(\mathbb R)$ and have, after extraction, a weak limit $\varphi$ in that space, a strong limit in every finite $L^q$ on compact intervals, and an almost-everywhere limit. Their restrictions agree with $\psi_{j,e}(t_j+\cdot)$ on every fixed compact interval for large $j$. The mass bound on $[0,1]$ gives $\|\varphi\|_{L^2(0,1)}\ge\sqrt\eta/2$, so $\varphi\ne0$.

We next justify that $\varphi$ solves the full-line equation, including the use of moving test functions. Fix $\zeta\in C_c^\infty(\mathbb R,\mathbb C^2)$. Put $\zeta_j(x)=\zeta(x-t_j)$ on $e$ and zero on all other edges. For large $j$, its support stays away from every vertex, so $\zeta_j\in\operatorname{dom}(\mathcal D_c)$. Moreover,
\[
\|\zeta_j\|_Y^2
=\langle|\mathcal D_c|\zeta_j,\zeta_j\rangle
\le\|\mathcal D_c\zeta_j\|_2\|\zeta_j\|_2
\le\bigl(c\|\zeta'\|_2+M\|\zeta\|_2\bigr)\|\zeta\|_2.
\]
This bound is independent of $j$. Thus $\langle\Phi'(\psi_j),\zeta_j\rangle\to0$. In the linear term, move $\mathcal D_c$ onto $\zeta_j$ using the form-domain identity; after translation, local convergence gives
\[
\operatorname{Re}\int_{\mathbb R}
\langle\varphi,(\mathcal D_\infty-\omega)\zeta\rangle\,dx
-\operatorname{Re}\int_{\mathbb R}\langle F(\varphi),\zeta\rangle\,dx=0.
\]
This is the full-line distributional equation. Since $\varphi\in H^{1/2}(\mathbb R)\subset L^2\cap L^{2p-2}$, we have $F(\varphi)\in L^2$; the equation therefore implies $\varphi'\in L^2$. In particular, $\varphi\in H^1(\mathbb R)$ and $\Phi_\infty'(\varphi)=0$.

Finally, fix $R,L>0$. For sufficiently large $j$, the compact subgraph $\mathcal K_R$ and the interval $[t_j-L,t_j+L]$ on $e$ are disjoint. Positivity of $\hat G$, \eqref{eq3.14}, and Fatou's lemma on these two sets yield
\[
d\ge\int_{\mathcal K_R}\hat G(|\psi|)\,dx
+\int_{-L}^{L}\hat G(|\varphi|)\,dx.
\]
Let $R,L\to\infty$ and use the critical-point identities for $\psi$ and $\varphi$. This proves \eqref{eq3.13}.
\end{proof}

\begin{Cor}\label{Cor3.9}
Under the preceding hypotheses, every Cerami sequence at a level $0<d<d_\infty$ has a subsequence converging strongly in $Y$ to a nontrivial bound state of action $d$.
More generally, if a Cerami sequence has weak limit $\psi$ and $d-\Phi(\psi)<d_\infty$, then that subsequence converges strongly to $\psi$.
\end{Cor}
\begin{proof}
The escape alternative contradicts the stated strict inequality. In the first assertion, strong convergence gives $\Phi(\psi)=d>0$, so $\psi\ne0$.
\end{proof}

\begin{Lem}\label{Lem3.10}
Assume $(g_1)$--$(g_5)$ and fix $m,c>0$, $|\omega|<mc^2$. If
\begin{equation}\label{eq3.15}
\Gamma_{c,\omega}(e)<d_c^\infty(\omega)
\quad\text{for some }e\in Y_c^+\setminus\{0\},
\end{equation}
then there is a nonzero bound state of action at most $\Gamma_{c,\omega}(e)$.
\end{Lem}
\begin{proof}
Lemma~\ref{Lem3.5} gives a Cerami sequence at $0<\kappa\le d\le\Gamma_{c,\omega}(e)<d_c^\infty(\omega)$. Corollary~\ref{Cor3.9} yields the required bound state.
\end{proof}
\subsection{Ground states below the line level}\label{Sec3.3}

In this subsection assume additionally that $g$ is nondecreasing.
For $r,s\ge0$ this implies
\begin{equation}\label{eq3.16}
 G(s)-G(r)\ge\frac12g(r)(s^2-r^2).
\end{equation}
Indeed, for $s\ge r$ integrate $g(t)t\ge g(r)t$ from $r$ to $s$;
for $s<r$ integrate the reverse comparison from $s$ to $r$.

\begin{Lem}\label{Lem3.11}
Fix parameters $m,c>0$ and $|\omega|<mc^2$. Every nonzero critical
point $\psi$ of $\Phi$ has $\psi^+\ne0$ and satisfies
\[
 \Phi(\psi)=\Gamma_{c,\omega}(\psi^+)
 :=\sup_{t\ge0,\,\eta\in Y^-}\Phi(t\psi^++\eta).
\]
More generally, the same conclusion holds if $\psi\notin Y^-$,
$\langle\Phi'(\psi),\psi\rangle=0$, and
$\Phi'(\psi)|_{Y^-}=0$.
\end{Lem}
\begin{proof}
Write $q(h,k)$ for the real quadratic-form pairing of
$\mathcal D_c-\omega$. If a nonzero critical point belonged to
$Y^-$, then
\[
 0=\langle\Phi'(\psi),\psi\rangle
 =q(\psi,\psi)-\int g(|\psi|)|\psi|^2<0,
\]
a contradiction. For $t\ge0$ and $\eta\in Y^-$, put
$z=t\psi^++\eta=t\psi+\zeta$, where
$\zeta=\eta-t\psi^-\in Y^-$. The stated criticality identities give
\[
 q(\psi,\psi)=\int g(|\psi|)|\psi|^2,
 \qquad q(\psi,\zeta)=\operatorname{Re}\int g(|\psi|)\psi\bar\zeta.
\]
Here spinor products denote the Euclidean Hermitian pairing and all
integrals are over $\mathcal G$. Since
$|z|^2=t^2|\psi|^2+2t\operatorname{Re}(\psi\cdot\bar\zeta)+|\zeta|^2$,
expansion yields the exact identity
\begin{align*}
 \Phi(z)-\Phi(\psi)
 &=\frac12q(\zeta,\zeta)-\frac12\int g(|\psi|)|\zeta|^2\\
 &\quad+\int\left[\frac12g(|\psi|)(|z|^2-|\psi|^2)
                   -G(|z|)+G(|\psi|)\right].
\end{align*}
By \eqref{eq3.16}, $g\ge0$, and the spectral gap,
\[
 \Phi(z)-\Phi(\psi)
 \le\frac12q(\zeta,\zeta)
 \le-\frac12\left(1-\frac{|\omega|}{mc^2}\right)\|\zeta\|_Y^2
 \le0.
\]
Equality of the supremum follows by choosing $t=1$, $\eta=\psi^-$.
\end{proof}

\begin{Lem}\label{Lem3.12}
Fix $m,c,\omega$ as above and put
\[
 d_c=\inf_{e\in Y^+\setminus\{0\}}\Gamma_{c,\omega}(e).
\]
If
$d_c<d_c^\infty(\omega)$, then $d_c$ is attained by a nonzero critical
point and is the least action among all nonzero critical points of
$\Phi$.
\end{Lem}
\begin{proof}
Choose $e_k\in Y^+\setminus\{0\}$ with
$\Gamma_{c,\omega}(e_k)\downarrow d_c$; after deleting finitely many terms,
$\Gamma_{c,\omega}(e_k)<d_c^\infty(\omega)$. Lemma~\ref{Lem3.5} and Corollary~\ref{Cor3.9}
give a critical point $\psi_k\ne0$ with
\[
 \kappa\le\Phi(\psi_k)=\ell_k\le\Gamma_{c,\omega}(e_k).
\]
By Lemma~\ref{Lem3.11}, every nonzero critical point
$\varphi$ satisfies $\Phi(\varphi)=\Gamma_{c,\omega}(\varphi^+)\ge d_c$.
In particular, $d_c\le\ell_k\le\Gamma_{c,\omega}(e_k)$, so
$\Phi(\psi_k)\to d_c$. Since $\Phi'(\psi_k)=0$, this is itself a
Cerami sequence at $d_c$. Moreover $d_c\ge\kappa>0$. Applying Corollary~\ref{Cor3.9} at $d_c<d_c^\infty(\omega)$ produces a critical point of
action $d_c$, which is nonzero because $d_c>0$. The same fiber
inequality proves that this action is minimal among all nonzero
critical points.
\end{proof}

\section{Nonrelativistic compactness and the sharp threshold}\label{Sec4}
Throughout this section $m>0$ and $\nu<0$ are fixed, and the parameters satisfy \eqref{eq1.4}--\eqref{eq1.6}. Constants may depend on these parameters, the graph, $g$, the action bound, the parameter sequences and, when used, the tightness modulus of the family, but not on $n$. Let $\mathcal H$ denote the closed subspace of the broken scalar $H^1$ space whose endpoint values agree at each vertex.

\begin{Pro}\label{Pro4.1}
Fix $m>0$ and assume $(g_1)$--$(g_5)$ for an arbitrary finite $p>2$.
Let $(c_n,\omega_n)$ satisfy \eqref{eq1.4}--\eqref{eq1.6}, and let
$\psi_n=(u_n,v_n)^T\in\operatorname{dom}(\mathcal D_n)$ be nontrivial solutions of \eqref{eq1.1} such that
\[
\sup_n\Phi_n(\psi_n)\le M<\infty.
\]
Then
\[
\sup_n\|\psi_n\|_{H^1(\mathcal G,\mathbb C^2)}<\infty,
\qquad \|v_n\|_{H^1}=\mathcal O(c_n^{-1}),
\qquad \inf_n\|u_n\|_{H^1}>0.
\]
Moreover, for the real differential of
\[
J(w)=\frac12\int_{\mathcal G}(|w'|^2-\nu|w|^2)\,dx
-2m\int_{\mathcal G}G(|w|)\,dx,
\qquad w\in\mathcal H,
\]
one has $dJ(u_n)\to0$ in $\mathcal H^*$, where $\mathcal H$ is the vertex-continuous scalar $H^1$ space.
\end{Pro}
\begin{proof}
Put
\[
\delta_n=mc_n^2-\omega_n,\quad K_n=mc_n^2+\omega_n,
\quad b_n=K_n/c_n^2,\quad a_n=\delta_n b_n.
\]
Then $\delta_n\to-\nu/(2m)>0$, $b_n\to2m$, and $a_n\to-\nu>0$; all these sequences are positive. In particular,
\[
a_*:=\inf_n a_n>0,
\qquad \sup_n(\omega_n/c_n^2+c_n^{-2})<\infty.
\]
For exact solutions the action identity is
\begin{equation}\label{eq4.1}
\Phi_n(\psi_n)=\int_{\mathcal G}\widehat G(|\psi_n|)\,dx.
\end{equation}
By $(g_4)$, $g\ge0$ and
\[
\widehat G(s)\ge\left(\frac12-\frac1\theta\right)g(s)s^2.
\]
Consequently,
\begin{equation}\label{eq4.2}
\int_{\mathcal G}g(|\psi_n|)|\psi_n|^2\,dx
\le\frac{2\theta M}{\theta-2}.
\end{equation}
Positivity and continuity of $\widehat G$ on compact subsets of $(0,\infty)$, together with $(g_5)$, give a number $a_1>0$ such that
$\widehat G(s)\ge a_1s^p$ for all $s\ge1$. Set
\[
h_n=(|\psi_n|-1)_+,
\qquad X_n=\|\psi_n'\|_{L^2},
\qquad q=\frac{p+2}{2}.
\]
The Sobolev chain rule gives $|h_n'|\le|\psi_n'|$ almost everywhere, and
\begin{equation}\label{eq4.3}
\|h_n\|_{L^p}^p
\le\int_{\{|\psi_n|>1\}}|\psi_n|^p\,dx
\le M/a_1.
\end{equation}

Here is the needed endpoint estimate without a uniform $L^2$ input.
On a bounded edge $I$ of length $L$, choose a point $y$ with
$h_n(y)\le L^{-1/p}\|h_n\|_{L^p(I)}$. The fundamental theorem of calculus and Cauchy--Schwarz give
\[
\begin{aligned}
\|h_n\|_{L^\infty(I)}^q
&\le L^{-q/p}\|h_n\|_{L^p(I)}^q
+q\int_I h_n^{q-1}|h_n'|\,dx\\
&\le L^{-q/p}\|h_n\|_{L^p(I)}^q
+q\|h_n\|_{L^p(I)}^{p/2}\|h_n'\|_{L^2(I)},
\end{aligned}
\]
because $2(q-1)=p$. On a half-line the same estimate holds without the first term, by integrating to infinity, where $h_n$ vanishes. Each individual $\psi_n$ is in $H^1$; this justifies these operations without assuming a uniform $L^2$ bound. As there are finitely many edges, \eqref{eq4.3} yields
\begin{equation}\label{eq4.4}
\|\psi_n\|_{L^\infty}
\le1+\|h_n\|_{L^\infty}
\le C\bigl(1+X_n^{2/(p+2)}\bigr).
\end{equation}
By $(g_3)$, nonnegativity of $g$, and \eqref{eq4.2},
\begin{equation}\label{eq4.5}
\begin{aligned}
\int_{\mathcal G}g(|\psi_n|)^2|\psi_n|^2\,dx
&\le\|g(|\psi_n|)\|_{L^\infty}
\int_{\mathcal G}g(|\psi_n|)|\psi_n|^2\,dx\\
&\le C\bigl(1+X_n^\alpha\bigr),
\qquad \alpha=\frac{2(p-2)}{p+2}<2.
\end{aligned}
\end{equation}

The exact Dirac norm identity, including the cancellation of vertex terms, is
\[
\|\mathcal D_n\psi_n\|_{L^2}^2
=c_n^2\|\psi_n'\|_{L^2}^2+m^2c_n^4\|\psi_n\|_{L^2}^2.
\]
Indeed, expansion leaves a cross term proportional to
$\operatorname{Re}\int i(u_n\overline{v_n})'$, which vanishes because $u_n$ is continuous at vertices, the signed traces of $v_n$ sum to zero, and the components vanish at infinite endpoints. Substituting
$\mathcal D_n\psi_n=(\omega_n+g(|\psi_n|))\psi_n$ gives
\begin{equation}\label{eq4.6}
X_n^2+a_n\|\psi_n\|_{L^2}^2
=\frac{2\omega_n}{c_n^2}\int_{\mathcal G}g(|\psi_n|)|\psi_n|^2\,dx
+\frac1{c_n^2}\int_{\mathcal G}g(|\psi_n|)^2|\psi_n|^2\,dx
\le C(1+X_n^\alpha).
\end{equation}
Since $\alpha<2$ and $a_n\ge a_*>0$, Young's inequality proves the uniform $H^1$ bound for every finite $p>2$.

We now prove the remaining assertions. The component equations are
\[
-ic_nv_n'+\delta_nu_n=g(|\psi_n|)u_n,
\qquad -ic_nu_n'-(K_n+g(|\psi_n|))v_n=0.
\]
Uniform boundedness of $g(|\psi_n|)$, boundedness of $\delta_n$, and $K_n+g\ge K_n$ imply
\[
\|v_n'\|_2\le Cc_n^{-1}\|u_n\|_2,
\qquad \|v_n\|_2\le(c_n/K_n)\|u_n'\|_2,
\qquad \|v_n\|_{H^1}\le Cc_n^{-1}\|u_n\|_{H^1}.
\]
For every $\phi\in\mathcal H$, integration by parts in the lower component has no vertex contribution and gives the exact identity
\begin{equation}\label{eq4.7}
\int_{\mathcal G}u_n'\overline{\phi'}+a_nu_n\overline\phi\,dx
=b_n\int_{\mathcal G}g(|\psi_n|)u_n\overline\phi\,dx
+\frac{i}{c_n}\int_{\mathcal G}g(|\psi_n|)v_n\overline{\phi'}\,dx.
\end{equation}
The last term is $\mathcal O(c_n^{-2})\|\phi\|_{\mathcal H}$. Moreover,
$\big||\psi_n|-|u_n|\big|\le|v_n|$ and $v_n\to0$ in $L^\infty$, so
\[
\|\{g(|\psi_n|)-g(|u_n|)\}u_n\|_2\to0.
\]
Taking real parts in \eqref{eq4.7}, and using $a_n\to-\nu$, $b_n\to2m$, proves $dJ(u_n)\to0$.

Finally, the spectral gap gives
\[
\delta_n\|\psi_n\|_2
\le\|(\mathcal D_n-\omega_n)\psi_n\|_2
\le\|g(|\psi_n|)\|_\infty\|\psi_n\|_2.
\]
Since $\psi_n\ne0$ and $\inf_n\delta_n>0$, the condition $g(0)=0$ gives a positive lower bound on $\|\psi_n\|_\infty$. As $v_n\to0$ in $L^\infty$, this gives a positive lower bound on $\|u_n\|_{H^1}$ for large $n$. For each of the finitely many remaining indices, $u_n\ne0$, since otherwise the second component equation and $K_n+g>0$ would give $v_n=0$. This proves the assertion for all $n$.
\end{proof}

In the remainder of this section write $\lambda=-\nu>0$ and $e_{\mathbb R}=e_{\mathbb R}(m,\nu)$, with the line profile and normalization of Proposition~\ref{Pro2.10}.

\begin{Pro}\label{Pro4.2}
Under the assumptions of Proposition~\ref{Pro4.1}, first pass to a subsequence such that
\[
\Phi_n(\psi_n)\to\ell,
\qquad u_n\rightharpoonup u\text{ in }\mathcal H,
\qquad u_n\to u\text{ locally uniformly on }\mathcal G.
\]
Then $u$ is a weak solution of the limiting NLS equation, possibly zero, and one of the following holds:
\begin{itemize}
\item $u_n\to u$ strongly in $\mathcal H$ and
$\ell=\int_{\mathcal G}\widehat G(|u|)\,dx$;
\item
\begin{equation}\label{eq4.8}
\ell\ge\int_{\mathcal G}\widehat G(|u|)\,dx+e_{\mathbb R}.
\end{equation}
In this case, after a further subsequence there are an external edge $e$, numbers $y_n\to\infty$, and a nonzero line solution $U$ such that
$u_{n,e}(y_n+\cdot)\to U$ locally uniformly and weakly in $H^1_{\rm loc}(\mathbb R)$.
\end{itemize}
\end{Pro}
\begin{proof}
Write $F(z)=g(|z|)z$. Uniform $H^1$ boundedness implies a common $L^\infty$ bound and boundedness of $F(u_n)$ in $L^2$. Local uniform convergence then gives
$F(u_n)\rightharpoonup F(u)$ in $L^2$: this follows first against compactly supported test functions and then against arbitrary $L^2$ tests by approximation. Passing to the limit in $dJ(u_n)\to0$ proves $dJ(u)=0$.

Set $w_n=u_n-u$. Subtracting $dJ(u)$ from $dJ(u_n)$ and testing with $w_n$ yields
\begin{equation}\label{eq4.9}
\|w_n'\|_2^2+\lambda\|w_n\|_2^2
=2m\operatorname{Re}\int_{\mathcal G}(F(u_n)-F(u))\overline{w_n}\,dx+o(1).
\end{equation}
We claim that failure of strong convergence implies
\begin{equation}\label{eq4.10}
\lim_{R\to\infty}\limsup_{n\to\infty}
\|u_n\|_{L^\infty(\mathcal G\setminus\mathcal K_R)}>0.
\end{equation}
Indeed, suppose the left-hand side were zero. Given $\varepsilon>0$, choose $r>0$ with $|g(s)|\le\varepsilon$ for $0\le s\le r$. Since $u$ tends to zero on each half-line, one can choose $R$ so that $|u|\le r$ outside $\mathcal K_R$ and, for all sufficiently large $n$, $|u_n|\le r$ there as well. On this complement,
\[
|(F(u_n)-F(u))\overline{w_n}|
\le\varepsilon(|u_n|+|u|)|w_n|,
\]
whose integral is at most $C\varepsilon$ by the uniform $L^2$ bounds. On $\mathcal K_R$ the integral tends to zero by local uniform convergence. Equation \eqref{eq4.9} therefore gives
$\limsup_n(\|w_n'\|_2^2+\lambda\|w_n\|_2^2)\le C\varepsilon$.
Letting $\varepsilon\downarrow0$ proves strong convergence and establishes the claim.

If convergence is not strong, \eqref{eq4.10} supplies a further subsequence, a fixed external edge $e$, points $y_n\to\infty$ on that edge, and $\eta>0$ such that
$|u_{n,e}(y_n)|\ge\eta$. The functions
\[
U_n(t)=u_{n,e}(y_n+t),\qquad t>-y_n,
\]
have uniformly bounded $H^1$ norms on every fixed compact interval. A diagonal extraction gives a function $U\in H^1(\mathbb R)$ such that
$U_n\rightharpoonup U$ in $H^1_{\rm loc}$ and $U_n\to U$ locally uniformly. The global $H^1$ bound for $U$ follows by lower semicontinuity on $[-L,L]$ and then $L\to\infty$; also $|U(0)|\ge\eta$.
For every $\varphi\in C_c^\infty(\mathbb R,\mathbb C)$, its translate to a neighborhood of $y_n$ on $e$, extended by zero to the rest of the graph, is an admissible test function in $\mathcal H$ for large $n$, with a norm independent of $n$. Testing $dJ(u_n)\to0$ with these functions and passing to the limit proves
\[
-U''+\lambda U=2m g(|U|)U\quad\text{on }\mathbb R.
\]
Thus $U$ is a nonzero line solution and
$\int_{\mathbb R}\widehat G(|U|)\,dt=e_{\mathbb R}$.

Fix $R,L>0$. For all sufficiently large $n$, the compact subgraph $\mathcal K_R$ and the interval $[y_n-L,y_n+L]$ on $e$ are disjoint. Since $\widehat G\ge0$, the exact action identity gives
\[
\Phi_n(\psi_n)
\ge\int_{\mathcal K_R}\widehat G(|\psi_n|)\,dx
+\int_{y_n-L}^{y_n+L}\widehat G(|\psi_{n,e}|)\,dx.
\]
The convergence $v_n\to0$ in $L^\infty$ and the respective local uniform convergences of $u_n$ imply
\[
\ell\ge\int_{\mathcal K_R}\widehat G(|u|)\,dx
+\int_{-L}^L\widehat G(|U|)\,dt.
\]
Letting $R,L\to\infty$ proves \eqref{eq4.8}.

If instead $u_n\to u$ strongly in $\mathcal H$, then
$\psi_n\to(u,0)^T$ strongly in $L^2$ and the sequence is uniformly bounded in $L^\infty$. On this common bounded range, $\widehat G'(s)=\tfrac12g'(s)s^2$ satisfies $|\widehat G'(s)|\le Cs$, and hence
\[
|\widehat G(s)-\widehat G(t)|\le C(s+t)|s-t|.
\]
Cauchy--Schwarz now proves convergence of the action to
$\int_{\mathcal G}\widehat G(|u|)\,dx$, completing the first alternative.
\end{proof}

\begin{Lem}\label{Lem4.3}
Under the hypotheses of Proposition~\ref{Pro4.1}, let the graph have nonempty compact core $\mathcal K$ and assume that the family $\{\psi_n\}$ is spatially tight in the sense of \eqref{eq2.20}. Then there are $C,\kappa>0$ such that
\[
|\psi_n(x)|\le C e^{-\kappa d(x,\mathcal K)}
\quad\text{for all }n\text{ and }x\in\mathcal G.
\]
\end{Lem}
\begin{proof}
Uniform $L^\infty$ boundedness follows from the $H^1$ bound. On each ray,
\[
|\psi_{n,e}(s)|^2
\le2\|\psi_{n,e}\|_{L^2(s,\infty)}\|\psi_{n,e}'\|_{L^2(s,\infty)},
\]
so tightness gives pointwise decay uniformly in $n$. 
Set $\mathcal M=-i\sigma_1\frac d{dx}$. On each edge,
\begin{equation}\label{eq4.11}
\mathcal M\psi_n=-mc_n\sigma_3\psi_n
+\frac{\omega_n+g(|\psi_n|)}{c_n}\psi_n.
\end{equation}
The equation and the Sobolev chain rule give $\psi_n\in H^2$ on each edge: on its bounded range, $g$ is Lipschitz, $g(0)=0$, and $g(|\psi_n|)\psi_n\in H^1$. Thus differentiation is justified in $L^2$ on each edge. Using $\sigma_1\sigma_3=-\sigma_3\sigma_1$ gives
\[
-\psi_n''+q_n\psi_n=-\frac{i}{c_n}\partial_x(g(|\psi_n|))\sigma_1\psi_n,
\]
where
\[
q_n:=m^2c_n^2-\frac{(\omega_n+g(|\psi_n|))^2}{c_n^2}
=a_n-\frac{2\omega_n}{c_n^2}g(|\psi_n|)-\frac{g(|\psi_n|)^2}{c_n^2}.
\]
The right-hand side has zero real scalar product with $\psi_n$, because $\sigma_1$ is Hermitian and $\partial_x(g(|\psi_n|))$ is real. To justify the modulus inequality at zeros, set $w_{n,\varepsilon}=(|\psi_n|^2+\varepsilon^2)^{1/2}$. The inequality
\[
w_{n,\varepsilon}''\ge
\frac{\operatorname{Re}(\psi_n''\cdot\overline{\psi_n})}{w_{n,\varepsilon}}
\]
holds almost everywhere. Substitution of the equation and passage to the limit $\varepsilon\downarrow0$ on compact intervals yield, in distributions,
\[
-w_n''+q_nw_n\le0,\qquad w_n=|\psi_n|.
\]

By $a_n\to-\nu>0$ and the positivity of every $a_n$, $a_*:=\inf_n a_n>0$. Moreover,
$B:=\sup_n2\omega_n/c_n^2<\infty$ and $D:=\sup_nc_n^{-2}<\infty$.
Choose $\eta>0$ with $B\eta+D\eta^2\le a_*/2$. By the uniform pointwise tails just proved and $g(0)=0$, there exists $R>0$, independent of $n$, such that $|g(|\psi_n(x)|)|\le\eta$ whenever $d(x,\mathcal K)\ge R$. Therefore
\[
q_n(x)\ge\tau:=a_*/2>0
\quad\text{for every }n\text{ and }d(x,\mathcal K)\ge R.
\]
Let $M_\infty=\sup_n\|\psi_n\|_{L^\infty}$ and define
\[
\Gamma(s)=M_\infty e^{-\sqrt\tau(s-R)},\qquad s\ge R.
\]
On each external edge, $-w_{n,e}''+\tau w_{n,e}\le0$, whereas $-\Gamma''+\tau\Gamma=0$ and $w_{n,e}(R)\le\Gamma(R)$. Testing the difference inequality with $(w_{n,e}-\Gamma)_+\in H^1_0((R,\infty))$ gives $w_{n,e}\le\Gamma$. Consequently, choosing $\kappa=\sqrt\tau$ and $C=M_\infty e^{\sqrt\tau R}$ also covers $d(x,\mathcal K)\le R$ and proves the stated estimate.
\end{proof}

\begin{Cor}\label{Cor4.4}
Assume $(g_1)$--$(g_5)$, \eqref{eq1.4}--\eqref{eq1.6}, and let $\psi_n=(u_n,v_n)^T$ be nonzero bound states on a fixed graph with nonempty compact core. If
\begin{equation}\label{eq4.12}
\limsup_n\Phi_n(\psi_n)<e_{\mathbb R}(m,\nu),
\end{equation}
then $\{\psi_n\}$ is relatively compact in $H^1$, every limit along a subsequence has the form $(u,0)^T$ with $u\ne0$ solving \eqref{eq1.2}, and the whole family satisfies the exponential estimate of Lemma~\ref{Lem4.3}.
\end{Cor}
\begin{proof}
The strict upper limit and finiteness of each action give a uniform action bound. Proposition~\ref{Pro4.1} applies. From every subsequence extract one with convergent action and weak $H^1$ limit. Its limiting action is below $e_{\mathbb R}(m,\nu)$, so Proposition~\ref{Pro4.2} excludes escape and gives strong convergence. The lower bound for $\|u_n\|_{H^1}$ makes the limit nonzero.

Relative compactness also gives uniform tightness. Otherwise some $R_j\to\infty$ and $n_j\to\infty$ would satisfy
\[
\int_{\mathcal G\setminus\mathcal K_{R_j}}|\psi_{n_j}|^2\,dx\ge\varepsilon>0.
\]
A strongly $L^2$ convergent subsequence contradicts this inequality, since its limit has vanishing tails. Lemma~\ref{Lem4.3} proves the decay assertion.
\end{proof}

The following examples show that the line threshold is sharp.
\begin{Pro}\label{Pro4.5}
The escape constants in Theorem~\ref{Thm1.4} and the strict threshold in Corollary~\ref{Cor4.4} are optimal for the class of graphs under consideration. There are exact nontrivial solutions with bounded action, action tending to $e_{\mathbb R}(m,\nu)$, and weak limit zero. These examples can be realized on a graph with nonempty compact core.
\end{Pro}
\begin{proof}
Let $S_{2k}$ have $2k$ outward-oriented half-lines, and let $U_{c,\omega}=(f,ih)^T$ be the profile from Proposition~\ref{Pro2.9}. For $a\ge0$, put $U_{c,\omega}(x+a)$ on $k$ edges and $U_{c,\omega}(x-a)$ on the other $k$ edges. At the common vertex the upper traces all equal $f(a)$, and the lower traces cancel because $h$ is odd. This is an exact solution $\psi^a$. Evenness of $|U_{c,\omega}|$ gives
\[
\Phi(\psi^a)=k\,d^\infty_c(\omega),\qquad
\|\psi^a\|_2^2=k\|U_{c,\omega}\|_{L^2(\mathbb R)}^2,
\qquad
\|\psi^a\|_{H^1}^2=k\|U_{c,\omega}\|_{H^1(\mathbb R)}^2.
\]
As $a\to\infty$, local convergence to zero and boundedness give weak convergence to zero in $H^1$ and in $Y$, while the global norms do not vanish. For fixed $c,\omega$, this is a noncompact sequence of critical points at level $k d^\infty_c(\omega)$.

Now take $(c_n,\omega_n)$ satisfying \eqref{eq1.4}--\eqref{eq1.6} and translations $a_n\to\infty$. Proposition~\ref{Pro2.10} supplies uniformly bounded norms and common exponential tails of the centered profiles. Thus the translated graph solutions converge locally to zero and weakly to zero in $H^1$, while
\[
\Phi_n(\psi_n^{a_n})\to k e_{\mathbb R}(m,\nu).
\]
For $k=1$, this gives equality in the nonrelativistic escape bound and the threshold in Corollary~\ref{Cor4.4}. Inserting a degree-two transmission vertex into one ray makes a bounded edge and hence a nonempty compact core. The upper traces match and the two signed lower traces cancel at the new vertex, so the operator and these solutions are unchanged. This proves the assertion on a graph with nonempty compact core.
\end{proof}

\medskip\noindent\textbf{Proof of Theorem~\ref{Thm1.4}.}
Part~(a) is Proposition~\ref{Pro3.8}; in the strong case the action of the limit equals $d>0$.
For part~(b), Proposition~\ref{Pro4.1} gives the uniform estimates and Proposition~\ref{Pro4.2} gives the alternative after extraction of a convergent action sequence. In the strong convergence alternative, the lower bound for $\|u_n\|_{H^1}$ gives $u\ne0$. A strongly $L^2$ convergent sequence is spatially tight: its limit has vanishing tails, the sufficiently late terms are close to it in $L^2$, and finitely many earlier terms can be covered by increasing the radius. Lemma~\ref{Lem4.3}, applied to this extracted sequence, gives the uniform exponential estimate.

For part~(c), represent the line as a star with two half-lines and insert a degree-two transmission vertex on one ray. A constant sequence of centered fixed-parameter profiles realizes strong convergence in~(a). The centered nonrelativistic profiles converge strongly by Proposition~\ref{Pro2.10} and realize the strong convergence alternative in~(b). Proposition~\ref{Pro4.5} gives the escaping sequences with weak limit zero and equality in both loss bounds.\hfill$\square$

\section{Ground state limits and geometric conditions}\label{Sec5}
We assume $(g_1)$--$(g_5)$ and impose $(g_6)$ where stated. In the nonrelativistic arguments below, $m>0$ and $\nu<0$ are fixed and the parameters satisfy \eqref{eq1.4}--\eqref{eq1.6}.
\subsection{Comparison of variational levels}\label{Sec5.1}
Let $L=-\Delta_{\mathcal G}$ be the Kirchhoff Laplacian, let $\mathcal H$ be the continuous Kirchhoff form space and define
\[
 s_{\mathcal G}=\inf_{w\in\mathcal H\setminus\{0\}}
                    \max_{t\ge0}J(tw),
 \qquad
 a(w)=\|w'\|_2^2-\nu\|w\|_2^2.
\]

\begin{Lem}\label{Lem5.1}
The map $w\mapsto\max_{t\ge0}J(tw)$ is continuous on
$\mathcal H\setminus\{0\}$ for the strong $H^1$ topology. Thus
\[
 s_{\mathcal G}=\inf_{w\in\operatorname{dom}(L)\setminus\{0\}}
                    \max_{t\ge0}J(tw).
\]
If $(g_6)$ also holds, every $u\ne0$ with $J'(u)[u]=0$ satisfies
$J(u)=\max_{t\ge0}J(tu)\ge s_{\mathcal G}$. Thus $s_{\mathcal G}$ equals the infimum of $J$ on its Nehari manifold, and every nonzero critical point has action at least $s_{\mathcal G}$.
\end{Lem}
\begin{proof}
Suppose $w_j\to w\ne0$ in $\mathcal H$. Then $a(w_j)\to a(w)>0$,
and convergence also holds in $L^q$ for every $2\le q\le\infty$.
The estimate $G(s)\ge C_0s^\theta-C_0s^2$ gives
\[
 J(tw_j)\le\left(\frac12a(w_j)+2mC_0\|w_j\|_2^2\right)t^2
              -2mC_0\|w_j\|_\theta^\theta t^\theta.
\]
Because $\|w_j\|_\theta^\theta\to\|w\|_\theta^\theta>0$ and
$\theta>2$, these upper bounds are negative for $t\ge T$, with
one $T$ valid for all sufficiently large $j$ and for $w$ itself.
All ray maxima therefore lie in $[0,T]$. The functions $w_j,w$ are
uniformly bounded in $L^\infty$, and $g$ is bounded on the resulting
range of $t|w_j|$. Since $G'(s)=g(s)s$,
\[
 \sup_{0\le t\le T}\left|
   \int G(t|w_j|)-\int G(t|w|)\right|
 \le C_T(\|w_j\|_2+\|w\|_2)\|w_j-w\|_2\to0.
\]
Thus $J(tw_j)\to J(tw)$ uniformly on $[0,T]$, proving continuity
of the maxima. The operator domain $\operatorname{dom}(L)$ is dense
in its form domain $\operatorname{dom}(L^{1/2})=\mathcal H$:
for example $\mathbf1_{[0,k]}(L)w\to w$ in the form norm, which is
equivalent to the $H^1$ norm. This proves the equality of the infima.

Assume now $(g_6)$. If $J'(u)[u]=0$ and $u\ne0$, then
$a(u)=2m\int g(|u|)|u|^2$. Hence
\[
 \frac{d}{dt}J(tu)=2mt\int
                  \bigl(g(|u|)-g(t|u|)\bigr)|u|^2.
\]
Monotonicity of $g$ makes this derivative nonnegative for $0<t<1$
and nonpositive for $t>1$. Therefore $t=1$ is a global ray maximum
and $J(u)\ge s_{\mathcal G}$. Every nonzero ray attains its positive maximum at an interior point, which belongs to the Nehari manifold. This proves the equality of the two infima.
\end{proof}

We next compare the Dirac fiber maxima with the Schr\"odinger ray maxima.
Let $L=-\Delta_{\mathcal G}$ denote the nonnegative Kirchhoff Laplacian and
fix $0\ne w\in\operatorname{dom}(L)$. In particular, $w$ is continuous at
vertices, is $H^2$ on every edge, and satisfies the Kirchhoff derivative
conditions. Retain $\delta_n=mc_n^2-\omega_n\to-\nu/(2m)>0$ and
$0<\omega_n<mc_n^2$. Set
\[
 z_n^w=\binom{w}{-iw'/(2mc_n)},\qquad
 e_n=P_n^+z_n^w,\qquad
 \|h\|_n=\bigl\||\mathcal D_n|^{1/2}h\bigr\|_2.
\]
The lower component of $z_n^w$ satisfies the Dirac Kirchhoff condition,
so $z_n^w\in\operatorname{dom}(\mathcal D_n)$. Define
\[
 \Gamma_n(e)=\sup_{t\ge0,\,\eta\in Y_n^-}\Phi_n(te+\eta),
 \qquad
 J(w)=\frac12\bigl(\|w'\|_2^2-\nu\|w\|_2^2\bigr)
       -2m\int_{\mathcal G}G(|w|)\,dx.
\]

\begin{Lem}\label{Lem5.2}
Assume $(g_1)$--$(g_5)$ with $2<p<\infty$. Then $e_n\ne0$ for all
sufficiently large $n$, and
\begin{equation}\label{eq5.1}
 \lim_{n\to\infty}\Gamma_n(e_n)
 =\frac1{2m}\max_{t\ge0}J(tw).
\end{equation}
For all sufficiently large $n$, every fiber point with
$\Phi_n(te_n+\eta)\ge0$ satisfies $t+\|\eta\|_n\le C_w$.
If $t_ne_n+\eta_n$ is a sequence of fiber maximizers, or has action
$\Gamma_n(e_n)+o(1)$, then $\|\eta_n\|_n\to0$, and every accumulation
point of $t_n$ maximizes $t\mapsto J(tw)$. 
\end{Lem}

\begin{proof}
Let $T=-i\sigma_1\partial_x$ with the same vertex domain as $\mathcal D_n$.
It is self-adjoint, $\sigma_3$ preserves its domain, and
$T\sigma_3=-\sigma_3T$. Hence
\[
 \mathcal D_n=c_nT+mc_n^2\sigma_3,\qquad
 \mathcal D_n^2=c_n^2T^2+m^2c_n^4,
 \qquad |\mathcal D_n|=(c_n^2T^2+m^2c_n^4)^{1/2}.
\]
The square identity is an operator identity on
$\operatorname{dom}(T^2)$: indeed $\mathcal D_nh\in\operatorname{dom}(T)$
if and only if $Th\in\operatorname{dom}(T)$, because
$\sigma_3h\in\operatorname{dom}(T)$.
For $c_n\ge1$, functional calculus for $T^2$ gives
\begin{equation}\label{eq5.2}
 |\mathcal D_n|\ge c_n|\mathcal D_{\rm ref}|,
 \qquad \mathcal D_{\rm ref}=T+m\sigma_3.
\end{equation}
Here $\|h\|_{\rm ref}=\||\mathcal D_{\rm ref}|^{1/2}h\|_2$. All these operators have the same form domain, with equivalent norms
for each fixed $n$. The fixed form-domain Sobolev embedding therefore
implies, for every $2\le q<\infty$,
\begin{equation}\label{eq5.3}
 \|h\|_q\le C_q\|h\|_{\rm ref}
 \le C_qc_n^{-1/2}\|h\|_n.
\end{equation}
The constants here are independent of $n$.

We first control the projection error. Direct calculation yields
\[
 (\mathcal D_n-\omega_n)z_n^w
 =\binom{\delta_nw+Lw/(2m)}{-i\delta_nw'/(2mc_n)}.
\]
This is bounded in $L^2$. Put $r_n=P_n^-z_n^w$ and
\[
 B_n(\eta)=\|\eta\|_n^2+\omega_n\|\eta\|_2^2,
 \qquad \eta\in Y_n^-.
\]
On the negative spectral subspace,
$|\mathcal D_n-\omega_n|=|\mathcal D_n|+\omega_n\ge mc_n^2+\omega_n$.
By spectral calculus,
\begin{equation}\label{eq5.4}
 B_n(r_n)
 \le\frac{\|(\mathcal D_n-\omega_n)z_n^w\|_2^2}
              {mc_n^2+\omega_n}
 =O(c_n^{-2}).
\end{equation}
Since $e_n=z_n^w-r_n$, equations
\eqref{eq5.3}--\eqref{eq5.4} and
$w'\in H^1(\mathcal G)$ give
\begin{equation}\label{eq5.5}
 e_n\longrightarrow(w,0)^T\quad\hbox{in every }L^q,
 \qquad 2\le q<\infty.
\end{equation}
In particular, $e_n\ne0$ eventually. Let
\[
 A_n=\|e_n\|_n^2-\omega_n\|e_n\|_2^2>0.
\]
Orthogonality of the spectral subspaces and a direct calculation show
\begin{align*}
 A_n
 &=\langle(\mathcal D_n-\omega_n)z_n^w,z_n^w\rangle+B_n(r_n)\\
 &=\delta_n\|w\|_2^2+
   \left(\frac1{2m}+\frac{\delta_n}{4m^2c_n^2}\right)\|w'\|_2^2
   +B_n(r_n)\\
 &\longrightarrow A:=\frac1{2m}
       \bigl(\|w'\|_2^2-\nu\|w\|_2^2\bigr)>0.
\end{align*}
Every pairing in this calculation is real, or equivalently its real
part can be taken.

For any $t\ge0$ and $\eta\in Y_n^-$, the exact fiber identity is
\begin{equation}\label{eq5.6}
 \Phi_n(te_n+\eta)
 =\frac12A_nt^2-\frac12B_n(\eta)
   -\int_{\mathcal G}G(|te_n+\eta|)\,dx.
\end{equation}
Since $G\ge0$, a nonnegative action implies
$B_n(\eta)\le A_nt^2$. If $t=0$, this implies $\eta=0$.
We show that all such $t$ are uniformly bounded for large $n$.
Otherwise there are $n_k\to\infty$, $t_k\to\infty$, and
$\eta_k\in Y_{n_k}^-$ with nonnegative action. Set
$h_k=\eta_k/t_k$. Then $\|h_k\|_{n_k}^2\le A_{n_k}$, so
\eqref{eq5.3} implies $h_k\to0$ in every finite $L^q$.
Consequently,
$e_{n_k}+h_k\to(w,0)^T$ in $L^\theta(K)$ on any compact subgraph $K$.
Choose $K$ with $\int_K|w|^\theta>0$. The consequence of $(g_4)$
$G(s)\ge C_0s^\theta-C_0s^2$ gives
\begin{align*}
 \frac1{t_k^2}\int_{\mathcal G}G(t_k|e_{n_k}+h_k|)\,dx
 &\ge C_0t_k^{\theta-2}\int_K|e_{n_k}+h_k|^\theta\,dx
       -C_0\int_K|e_{n_k}+h_k|^2\,dx\longrightarrow+\infty.
\end{align*}
This contradicts \eqref{eq5.6} divided by $t_k^2$.
It follows that $t$ is uniformly bounded and, since
$\|\eta\|_n^2\le B_n(\eta)\le A_nt^2$, so is $\|\eta\|_n$.
To justify choosing $n_k\to\infty$ in this contradiction, recall that
for every fixed $n$ the same two-case normalized argument as in
Lemma~\ref{Lem3.2} makes the nonnegative part of this fiber bounded.
Thus finitely many indices cannot cause failure of an eventual
uniform bound.

Put $H(t)=\tfrac12At^2-\int_{\mathcal G}G(t|w|)\,dx=J(tw)/(2m)$.
Assumptions $(g_2)$--$(g_4)$ imply that $H(t)>0$ for small positive
$t$ and $H(t)\to-\infty$ as $t\to\infty$; hence $H$ attains a
positive maximum. For every fixed $t$, the growth bound on $G$ and
\eqref{eq5.5} give
\[
 \Phi_n(te_n)\longrightarrow H(t).
\]
Thus $\liminf_n\Gamma_n(e_n)\ge\max_{t\ge0}H(t)>0$.
The bounds above for fiber points of nonnegative action also imply
$\sup_n\Gamma_n(e_n)<\infty$ after discarding finitely many indices.
Choose fiber points with
$\Phi_n(t_ne_n+\eta_n)\ge\Gamma_n(e_n)-1/n$. They have positive
action eventually. After extraction, $t_n\to t_*$ and
$\eta_n\to0$ in every finite $L^q$ by \eqref{eq5.3}.
Therefore $t_ne_n+\eta_n\to(t_*w,0)^T$ in $L^2\cap L^p$ and
\[
 \int_{\mathcal G}G(|t_ne_n+\eta_n|)\,dx
 \longrightarrow\int_{\mathcal G}G(t_*|w|)\,dx.
\]
Here continuity of $G$ and $|G(s)|\le C(s^2+s^p)$ suffice; one may
extract a summably convergent subsequence in both spaces and apply
dominated convergence, then use the subsequence criterion.
Dropping $B_n(\eta_n)\ge0$ in \eqref{eq5.6} yields
\[
 \limsup_n\Gamma_n(e_n)\le H(t_*)\le\max_{t\ge0}H(t).
\]
Applied to a subsequence realizing the upper limit, this proves
\eqref{eq5.1}. For any maximizing sequence with error
$o(1)$, the same calculation and the exact identity show that
$H(t_*)=\max H$ and $B_n(\eta_n)\to0$ on every convergent
subsequence of $t_n$. Compactness of the possible $t_n$ then gives
$\|\eta_n\|_n\to0$ for the full sequence.
\end{proof}

\begin{Rek}\label{Rek5.3}
For each fixed $n$ with $e_n\ne0$, the supremum $\Gamma_n(e_n)$ is
attained. Indeed, the nonnegative part of the fiber is bounded by the
fixed-parameter argument of Lemma~\ref{Lem3.2}. Extract $t_j\to t$
and $\eta_j\rightharpoonup\eta$ in $Y_n^-$. Local compactness gives
almost-everywhere convergence along a further subsequence. The
quadratic negative term is weakly upper semicontinuous, while Fatou's
lemma, using $G\ge0$, gives weak lower semicontinuity of its nonlinear
integral. Hence the action is weakly upper semicontinuous on the
fiber, and a maximizing sequence converges to a maximizer in this
sense. 
\end{Rek}

\subsection{Scalar ground levels and compactness}\label{Sec5.2}
\begin{Pro}\label{Pro5.4}
Assume $(g_1)$--$(g_6)$ on a fixed finite connected noncompact graph with nonempty compact core, and let \eqref{eq1.4}--\eqref{eq1.6} hold. The variational levels $d_n=\inf_{e\in Y_{c_n}^+\setminus\{0\}}\Gamma_{c_n,\omega_n}(e)$ satisfy
\begin{equation}\label{eq5.7}
\limsup_n d_n\le\frac{s_{\mathcal G}}{2m}.
\end{equation}
Let $\psi_n=(u_n,v_n)^T$ be nonzero bound states such that
\begin{equation}\label{eq5.8}
\limsup_n\Phi_n(\psi_n)\le\frac{s_{\mathcal G}}{2m}.
\end{equation}
If a subsequence converges strongly in $H^1$ to $(u,0)^T$ with $u\ne0$, then $u$ is a Schr\"odinger action ground state and $J(u)=s_{\mathcal G}$.

If $g(s)=s^{p-2}$, $p\ge3$, and $\mathcal G$ satisfies condition~(H) but is isometric neither to the line nor to an exceptional tower of bubbles, then
\[
\Phi_n(\psi_n)\longrightarrow e_{\mathbb R}(m,\nu),
\qquad \psi_n\rightharpoonup0\quad\text{in }H^1(\mathcal G,\mathbb C^2).
\]
From every subsequence one can extract a further subsequence $n_k$, a half-line $e$ and points $x_k\to\infty$ on $e$ such that $u_{n_k,e}(x_k+\cdot)\to Q$ locally uniformly on $\mathbb R$, where $Q$ is a nonzero full-line solution of \eqref{eq1.2}, and
\[
\lim_{R\to\infty}\limsup_{k\to\infty}
\int_{\mathcal G\setminus I_{k,R}}
\hat G(|\psi_{n_k}|)\,dx=0,
\qquad I_{k,R}=\{x\in e:|x-x_k|<R\}.
\]
\end{Pro}
\begin{proof}
For every $0\ne w\in\operatorname{dom}(L)$, Lemma~\ref{Lem5.2} gives
\[
\limsup_n d_n\le\frac1{2m}\max_{t\ge0}J(tw).
\]
Taking the infimum and applying Lemma~\ref{Lem5.1} proves \eqref{eq5.7} on every graph under consideration.

By \eqref{eq5.8}, the actions are bounded, so Proposition~\ref{Pro4.1} applies. If $\psi_{n_k}\to(u,0)^T$ strongly in $H^1$ with $u\ne0$, then $u$ solves \eqref{eq1.2}. On the common bounded range, $\hat G'(s)=g'(s)s^2/2$ gives
\[
|\hat G(s)-\hat G(t)|\le C(s+t)|s-t|.
\]
Strong $L^2$ convergence and the critical-point identities therefore yield
\[
\Phi_{n_k}(\psi_{n_k})\longrightarrow
\int_{\mathcal G}\hat G(|u|)\,dx
=\frac{J(u)}{2m}\le\frac{s_{\mathcal G}}{2m}.
\]
Lemma~\ref{Lem5.1} gives $J(u)\ge s_{\mathcal G}$, proving equality.

For the power nonlinearity, put $\lambda=-\nu$, $b=(2m)^{-1/(p-2)}$ and
\[
J_\lambda(w)=\frac12\int_{\mathcal G}(|w'|^2+\lambda|w|^2)\,dx
-\frac1p\int_{\mathcal G}|w|^p\,dx.
\]
Let $\mathcal H_{\mathbb R}$ denote the real-valued subspace of $\mathcal H$, and set
\[
\mathcal N_\lambda(\mathcal G)
=\{w\in\mathcal H_{\mathbb R}\setminus\{0\}:J_\lambda'(w)[w]=0\}.
\]
Each nonzero real ray has a unique Nehari point, at which its action is maximal. For complex $w$, the function $|w|$ belongs to $\mathcal H_{\mathbb R}$ and satisfies $\|(|w|)'\|_2\le\|w'\|_2$. Consequently,
\[
\max_{t\ge0}J_\lambda(t|w|)\le\max_{t\ge0}J_\lambda(tw).
\]
The infima over real and complex rays are thus equal. A complex minimizer would also yield a real Nehari minimizer by taking its modulus and maximizing on that ray. Since $J(bw)=b^2J_\lambda(w)$, we obtain
\begin{equation}\label{eq5.9}
s_{\mathcal G}=b^2\inf_{\mathcal N_\lambda(\mathcal G)}J_\lambda,
\qquad 2m e_{\mathbb R}(m,\nu)=b^2s_\lambda,
\end{equation}
where $s_\lambda$ is the full-line Nehari level for $J_\lambda$.

Under condition~(H), \cite[Theorem~2.7]{DeCoster2023} and \eqref{eq5.9} give
\[
s_{\mathcal G}=2m e_{\mathbb R}(m,\nu).
\]
The same theorem excludes attainment on the stated nonexceptional graphs, also in the complex space by the modulus argument above. Apply Theorem~\ref{Thm1.4}(b) to any subsequence, extracting a convergent action sequence and a weak limit $u$. The strong convergence alternative would contradict nonattainment by the first part of this proposition. The escape alternative gives
\[
e_{\mathbb R}(m,\nu)\ge\ell
\ge\int_{\mathcal G}\hat G(|u|)\,dx+e_{\mathbb R}(m,\nu).
\]
Thus $\ell=e_{\mathbb R}(m,\nu)$ and $u=0$, since $\hat G(s)>0$ for $s>0$. Applying this argument to every subsequence proves the asserted action convergence and weak convergence for the whole sequence.

The escaping profile supplied by Proposition~\ref{Pro4.2} has action $e_{\mathbb R}(m,\nu)$ by Proposition~\ref{Pro2.10}. For every fixed $R>0$, local uniform convergence and $v_{n_k}\to0$ in $H^1$ give
\[
\int_{I_{k,R}}\hat G(|\psi_{n_k}|)\,dx
\longrightarrow\int_{-R}^R\hat G(|Q|)\,dx.
\]
Subtract this equality from the total action limit and let $R\to\infty$. This proves that the profile carries all the action.
\end{proof}

\begin{Rek}\label{Rek5.5}
Estimate \eqref{eq5.7} concerns the variational levels regardless of their attainment. Proposition~\ref{Pro5.4} applies to a family of action ground states whenever such a family exists, since its actions equal $d_n$. Its necessary condition for compactness is attainment of the scalar level. The geometric conditions in Theorem~\ref{Thm1.1} and Corollary~\ref{Cor1.3} provide sufficient conditions for this attainment.
\end{Rek}

\begin{Pro}\label{Pro5.6}
Assume $(g_1)$--$(g_6)$ on a fixed graph with nonempty compact core, fix $m>0$ and $\nu<0$, and suppose
\begin{equation}\label{eq5.10}
s_{\mathcal G}<2m e_{\mathbb R}(m,\nu).
\end{equation}
For every parameter sequence \eqref{eq1.4}--\eqref{eq1.6}, the Dirac equation has action ground states at the levels
\[
d_n=\inf_{e\in Y_{c_n}^+\setminus\{0\}}\Gamma_{c_n,\omega_n}(e)
\]
for all sufficiently large $n$. Every choice of these states is relatively compact in $H^1$, has uniform exponential decay, and has only nonzero Schr\"odinger ground states as its limits along subsequences. Moreover,
\[
d_n\to\frac{s_{\mathcal G}}{2m},\qquad J(u)=s_{\mathcal G}
\]
for every such limiting state $u$.
\end{Pro}
\begin{proof}
Combining \eqref{eq5.7} with \eqref{eq5.10} gives
\[
\limsup_n d_n\le\frac{s_{\mathcal G}}{2m}<e_{\mathbb R}(m,\nu).
\]
Proposition~\ref{Pro2.10} gives $d^\infty_{c_n}(\omega_n)\to e_{\mathbb R}(m,\nu)$. Hence $d_n<d^\infty_{c_n}(\omega_n)$ for all sufficiently large $n$. For each such $n$, Lemma~\ref{Lem3.12} supplies a nonzero action ground state $\psi_n$ at level $d_n$. Corollary~\ref{Cor4.4} applies to every choice of these states by \eqref{eq5.7} and \eqref{eq5.10}.

Along any strongly convergent subsequence $\psi_n\to(u,0)^T$, the common $L^\infty$ bound and strong $L^2$ convergence give
\[
d_n=\int_{\mathcal G}\hat G(|\psi_n|)\,dx\longrightarrow
\int_{\mathcal G}\hat G(|u|)\,dx=\frac{J(u)}{2m}.
\]
Indeed, $\hat G'(s)=g'(s)s^2/2$ implies $|\hat G(s)-\hat G(t)|\le C(s+t)|s-t|$ on the common bounded range. The last equality uses $J'(u)[u]=0$. Lemma~\ref{Lem5.1} gives $J(u)\ge s_{\mathcal G}$, while \eqref{eq5.7} gives the reverse inequality. Thus $J(u)=s_{\mathcal G}$: the limit $u$ is an action ground state and also a least action solution of the NLS equation. Applying this argument to every subsequence proves convergence of the entire sequence of levels.\end{proof}
\subsection{Geometric conditions for ground states}\label{Sec5.3}
A pendant path is a terminal metric interval $[0,L]$ whose interior has no branching and which meets the rest of the graph only at $L$. Purely transmissive degree-two vertices along it may be suppressed. If an edge is reoriented, the corresponding spinor is transformed by $(u(x),v(x))\mapsto(u(\ell-x),-v(\ell-x))$; this preserves the equation and the vertex conditions.

\begin{Lem}\label{Lem5.7}
Assume $(g_1)$--$(g_6)$ and fix $m>0$, $\nu<0$. There is $L_*=L_*(m,\nu,g)>0$ such that every graph with a pendant path of length at least $L_*$ satisfies
\[
s_{\mathcal G}<2m e_{\mathbb R}(m,\nu).
\]
The choice of $L_*$ is independent of the rest of the graph.
\end{Lem}
\begin{proof}
Let $q_0$ be the even line profile in Proposition~\ref{Pro2.10}. Choose smooth cutoffs $\chi_R$ on $[0,\infty)$ which equal one on $[0,R]$, vanish on $[R+1,\infty)$ and have derivatives of order at most two bounded independently of $R$. Put $w_R=\chi_Rq_0$ on the pendant, measured from its terminal endpoint, and extend it by zero to the graph. If $L\ge R+2$, this function belongs to the domain of the Kirchhoff Laplacian: $q_0'(0)=0$, it vanishes near the attachment point, and transmission holds at suppressed degree-two vertices.

On the half-line, $w_R\to q_0$ in $H^1$. The ray-maximum continuity proved in Lemma~\ref{Lem5.1} therefore gives
\[
\max_{t\ge0}J_{\mathcal G}(tw_R)
\longrightarrow\max_{t\ge0}J_{\mathbb R_+}(tq_0)
=J_{\mathbb R_+}(q_0)=m e_{\mathbb R}(m,\nu).
\]
The middle equality follows from monotonicity of $g$ and the Nehari identity for $q_0$ on the half-line. Hence for all sufficiently large $R$,
\[
s_{\mathcal G}\le\max_{t\ge0}J_{\mathcal G}(tw_R)<2m e_{\mathbb R}(m,\nu).
\]
The choice of $R$ uses only the scalar half-line profile and the cutoffs, and is independent of the rest of the graph. Take $L_*=R+2$.
\end{proof}

\medskip\noindent\textbf{Proof of Theorem~\ref{Thm1.1}.}
Lemma~\ref{Lem5.7} gives the scalar strict comparison from the assumed pendant length. Proposition~\ref{Pro5.6} supplies the Dirac ground states, the convergence of their levels, and the relative compactness and identification of their nonrelativistic limits. Proposition~\ref{Pro4.1} gives~(i). The relative compactness supplied by Proposition~\ref{Pro5.6} gives spatial tightness, as in the proof of Corollary~\ref{Cor4.4}; Lemma~\ref{Lem4.3} then gives~(iii). These estimates apply to any choice of ground states at the levels $d_n$. Lemma~\ref{Lem3.12} also shows that these ground states are least action solutions.\hfill$\square$

We obtain an integral formula for a sufficient pendant length by using a mixed Neumann--Dirichlet profile. The following calculation is for the normalized scalar equation.

\begin{Pro}\label{Pro5.8}
Let $p>2$ and set
\[
 \kappa_p=\frac{p-2}{2p},\qquad
 Q_p(x)=\left(\frac p2\right)^{1/(p-2)}
 \operatorname{sech}^{2/(p-2)}\!\left(\frac{p-2}{2}x\right),
 \qquad
 \mathcal E_p=\kappa_p\int_{\mathbb R}Q_p^p\,dx.
\]
For every $\ell>0$ there is a unique positive solution $w_\ell$ of
\[
 -w''+w=w^{p-1}\quad\hbox{on }(0,\ell),\qquad
 w'(0)=0,\quad w(\ell)=0.
\]
It is strictly decreasing on $(0,\ell]$. Define
\[
 S_p(\ell)=\max_{t\ge0}
 \left\{\frac{t^2}{2}\int_0^\ell(w_\ell'^2+w_\ell^2)\,dx
       -\frac{t^p}{p}\int_0^\ell w_\ell^p\,dx\right\}.
\]
For $\beta>1$, put
\begin{align}
 D_\beta(s)&=\beta(1-s^p)-(1-s^2),\nonumber\\
 \ell(\beta)&=\int_0^1\frac{ds}{\sqrt{D_\beta(s)}},
 \label{eq5.11}\\
 S_p(\ell(\beta))&=\kappa_p
 \left(\frac{p\beta}{2}\right)^{p/(p-2)}
 \int_0^1\frac{s^p\,ds}{\sqrt{D_\beta(s)}}.
 \label{eq5.12}
\end{align}
The function $S_p$ is continuous and strictly decreasing, with
\[
 \lim_{\ell\downarrow0}S_p(\ell)=+\infty,
 \qquad \lim_{\ell\to\infty}S_p(\ell)=\frac{\mathcal E_p}{2}.
\]
Consequently there is a unique $\ell_p>0$ satisfying
$S_p(\ell_p)=\mathcal E_p$, and $S_p(\ell)<\mathcal E_p$ holds
exactly when $\ell>\ell_p$. Equations
\eqref{eq5.11}--\eqref{eq5.12} give an exact
quadrature definition of $\ell_p$.
\end{Pro}
\begin{proof}
If $A=w(0)$, the first integral is
\[
 w'^2=w^2-\frac2p w^p+C,
 \qquad C=\frac2p A^p-A^2.
\]
Any positive solution reaching zero has $C=w'(\ell)^2>0$;
otherwise uniqueness for the initial-value problem at $\ell$ would
give the zero solution. Hence $A>(p/2)^{1/(p-2)}$ and
$\beta=(2/p)A^{p-2}>1$. Since $A>1$, one has $w''(0)<0$.
Since $D_\beta(s)>0$ for $0\le s<1$,
the solution decreases from $A$ to its first zero, and separation of
variables gives \eqref{eq5.11}. Conversely this
quadrature constructs that solution. The zero of $D_\beta$ at $1$
is simple, with slope $2-p\beta<0$. Differentiation therefore gives
\[
 \ell'(\beta)=-\frac12\int_0^1
 \frac{1-s^p}{D_\beta(s)^{3/2}}\,ds<0.
\]
The integral is locally uniformly finite for $\beta>1$.
Moreover $\ell(\beta)\to\infty$ as $\beta\downarrow1$, since
$D_1(s)=s^2-s^p$, and $\ell(\beta)\to0$ as $\beta\to\infty$.
Thus every $\ell>0$ determines exactly one profile.

Testing its equation with $w_\ell$ gives
$\int_0^\ell(w_\ell'^2+w_\ell^2)=\int_0^\ell w_\ell^p$.
The ray maximum is consequently attained uniquely at $t=1$ and
equals $\kappa_p\int_0^\ell w_\ell^p$, which proves
\eqref{eq5.12}.
The nonlinearity $|w|^{p-2}w$ is $C^1$, and the first zero is
transverse because $C>0$. The $C^1$ dependence for the initial-value
problem, the implicit function theorem at this zero, and
$\ell'(\beta)<0$ therefore give $C^1$ dependence on $\ell$.
In particular, differentiating $w_\ell(\ell)=0$ yields
$\partial_\ell w_\ell(\ell)=-w_\ell'(\ell)$. Differentiating the
action and integrating by parts, using the equation and
$w_\ell'(0)=0$, now gives
\[
 S_p'(\ell)=\frac12w_\ell'(\ell)^2
       +w_\ell'(\ell)\partial_\ell w_\ell(\ell)
       =-\frac12w_\ell'(\ell)^2<0.
\]
These differentiations may equivalently be made after rescaling the
interval to $[0,1]$.

As $\beta\downarrow1$, dominated convergence in
\eqref{eq5.12}, with majorant
$s^p/\sqrt{s^2-s^p}$, gives $S_p(\ell(\beta))\to\mathcal E_p/2$;
the limiting quadrature is that of $Q_p$ on $\mathbb R_+$.
As $\beta\to\infty$, factoring out $\beta$ from $D_\beta$ shows
that the right-hand side of \eqref{eq5.12} is asymptotic
to a positive constant times
$\beta^{(p+2)/(2(p-2))}$, and hence diverges. Continuity and strict
monotonicity prove the assertion.
\end{proof}

\begin{Rek}\label{Rek5.9}
The full-line constant can also be written as
\[
 \mathcal E_p=\frac1p\left(\frac p2\right)^{p/(p-2)}
 \mathrm B\!\left(\frac12,\frac{p}{p-2}\right),
\]
where $\mathrm B$ is Euler's beta function. In particular,
$Q_4(x)=\sqrt2\operatorname{sech}x$ and $\mathcal E_4=4/3$.
For illustration, numerical evaluation of
\eqref{eq5.11}--\eqref{eq5.12} gives
$\ell_3\approx1.71$, $\ell_4\approx1.24$, $\ell_5\approx1.05$,
and $\ell_6\approx0.94$.
The threshold $\ell_p$ is a sufficient threshold obtained from this
test family.
\end{Rek}

\medskip\noindent\textbf{Proof of Corollary~\ref{Cor1.2}.}
Put $a=\sqrt{-\nu}$ and $b=((-\nu)/(2m))^{1/(p-2)}$. If $w$ is a scalar function on the graph $a\mathcal G$, obtained by multiplying every edge length by $a$, set $u(x)=b w(ax)$. A change of variables gives
\begin{equation}\label{eq5.13}
J_{\mathcal G}(u)=b^2a\,I_{a\mathcal G}(w),\qquad
I_{a\mathcal G}(w)=\frac12\int_{a\mathcal G}(|w'|^2+|w|^2)\,dy
-\frac1p\int_{a\mathcal G}|w|^p\,dy.
\end{equation}
The same factor gives $2m e_{\mathbb R}(m,\nu)=b^2a\mathcal E_p$.
On the scaled pendant of length $\ell=aL$, place the profile $w_\ell$ of Proposition~\ref{Pro5.8} and extend it by zero to the rest of the graph. The extension is continuous and belongs to the scalar form space $\mathcal H$. Its ray maximum is $S_p(\ell)$. Thus, if $\ell>\ell_p$, then
\[
s_{\mathcal G}\le b^2a S_p(\ell)
<b^2a\mathcal E_p=2m e_{\mathbb R}(m,\nu).
\]
Proposition~\ref{Pro5.6} gives all the conclusions of Theorem~\ref{Thm1.1}. The trial function is used in the form-domain definition of $s_{\mathcal G}$; its derivative need not satisfy Kirchhoff balance at the attachment point. The operator-domain approximation needed for the Dirac fiber comparison is supplied by Lemma~\ref{Lem5.1}. The scaling factor in \eqref{eq5.13} cancels from the strict comparison, so $\ell_p$ is independent of $m$. Every terminal vertex admits a pendant interval of positive length, which proves the last assertion.\hfill$\square$

\medskip\noindent\textbf{Proof of Corollary~\ref{Cor1.3}.}
For $\lambda=-\nu$ and an empty Dirichlet vertex set ($Z=\emptyset$ in the notation of the reference), \cite[Theorem~4.8 and Figure~6]{DeCoster2026} gives
\[
\inf_{\mathcal N_\lambda(\mathcal G)}J_\lambda<s_\lambda
\]
for each of the four graphs and every choice of positive bounded-edge lengths. By \eqref{eq5.9}, multiplication by $(2m)^{-2/(p-2)}$ yields $s_{\mathcal G}<2m e_{\mathbb R}(m,\nu)$. Proposition~\ref{Pro5.6} gives the stated conclusions.\hfill$\square$

We conclude with a geometric existence result for fixed $m,c>0$ and frequencies in a compact subinterval of the spectral gap.
\begin{Cor}\label{Cor5.10}
Assume in addition that $g$ is nondecreasing. Fix $m,c>0$ and a compact frequency interval $I\Subset(-mc^2,mc^2)$. There exists $L_I>0$ such that a finite connected noncompact graph with nonempty compact core containing a pendant path of length at least $L_I$ has an action ground state at every frequency $\omega\in I$.
\end{Cor}
\begin{proof}
Apply \cite[Theorem~1.1 and Proposition~2.15]{Yang2026} with $f=g$ and $\lambda=-\omega$. Its assumptions follow from $(g_1)$--$(g_4)$ and monotonicity. The theorem gives, uniformly for $\omega\in I$,
\[
\inf_{e\in Y^+\setminus\{0\}}\Gamma_{c,\omega}(e)<d^\infty_c(\omega).
\]
Lemma~\ref{Lem3.12} now proves attainment by an action ground state. The cited length threshold depends on the fixed $c$; we do not use it to pass to $c\to\infty$.
\end{proof}

\section*{Acknowledgment}
This work was supported by the National Natural Science Foundation of China (12301145, 12561020, 12261107) and Yunnan Fundamental Research Projects (202401AU070123, 202601AT070048).
M. Ruzhansky is also supported by FWO Odysseus 1 Grant G.0H94.18N: Analysis and Partial Differential Equations and the Methusalem programme of the Ghent University Special Research Fund (BOF) (Grant Number 01M01021). 

{\bf Data availability:}  Data sharing is not applicable to this article as no new data were created or analyzed in this study.

{\bf Conflict of Interests:} The authors declare that there is no conflict of interest.

\end{document}